\documentclass{mystyle}
\usepackage{graphicx}

\usepackage{amsmath,amsfonts,amssymb}
\usepackage{bm}
\usepackage{todonotes}
\usepackage{pgfplotstable, pgfplots}
\usepackage{booktabs}

\usepackage{tikz}
\usetikzlibrary{shapes,arrows,snakes,calendar,matrix,backgrounds,folding,calc,positioning,spy,decorations.pathreplacing}

\usepackage{siunitx}
\newcommand{\numeoc}[1]{\num[round-precision=1,round-mode=places, scientific-notation=false]{#1}}

\newcommand{\divergence}{\operatorname{div}}

\newcommand{\id}{\operatorname{Id}}

\newcommand{\curl}{\operatorname{curl}}

\newcommand{\rr}{\mathbb{R}}

\newcommand{\vecb}[1]{{#1}}
\newcommand{\matb}[1]{{#1}}

\newcommand{\trans}{{\textrm{T}}}

\newcommand{\ecoord}{{\vecb{x}}}

\newcommand{\normal}{{\vecb{n}}}
\newcommand{\tangential}{{\vecb{t}}}

\newcommand{\normalhat}{\hat{\vecb{n}}}
\newcommand{\tangentialhat}{\hat{\vecb{t}}}

\newcommand{\ds}{\mathop{~\mathrm{d} s}}

\newcommand{\dx}{\mathop{~\mathrm{d} \ecoord}}
\newcommand{\dxhat}{\mathop{~\mathrm{d} \hat{\ecoord}}}

\newcommand{\mesh}{\mathcal{T}_h}

\newcommand{\facetT}{\mathcal{F}_T}
\newcommand{\facets}{\mathcal{F}_h}
\newcommand{\facetsint}{\mathcal{F}_h^{\text{int}}}
\newcommand{\facetsext}{\mathcal{F}_h^{\text{ext}}}

\newcommand{\trace}[1]{\textrm{tr}({#1})}
\newcommand{\jump}[1]{ {[\![ #1 ] \!]}}
\newcommand{\Dev}[1]{\mathop{\text{dev}}{(#1)}}
\newcommand{\dev}[1]{\Dev{#1}}

\newcommand{\devnorm}[1]{ \left\|{#1}\right\|_{1, \text{dev}, h}}

\newcommand{\visc}{\nu}

\newcommand{\blfone}{b_{1}}
\newcommand{\ablf}{a}
\newcommand{\blftwo}{b_{2}}
\newcommand{\blfbig}{B}
\newcommand{\Velspace}{{\vecb{V}}}
\newcommand{\Velspaceh}{{\Velspace_h}}
\newcommand{\Velspacehdivfree}{{\Velspace_h^0}}

\newcommand{\Velvar}{{\vecb{u}}}
\newcommand{\Velvartest}{{\vecb{v}}}
\newcommand{\Velvarh}{{\vecb{u}_h}}
\newcommand{\Velvarhtest}{\vecb{v}_h}

\newcommand{\Velnormh}[1]{|| #1 ||_\Velspaceh}
\newcommand{\honenormh}[1]{|| #1 ||_{1,h} }

\newcommand{\Stressspace}{{\matb{\Sigma}}}
\newcommand{\Stressspaceh}{{\Stressspace_h}}
\newcommand{\Stressvar}{{\matb{\sigma}}}
\newcommand{\Stressvartest}{{\matb{\tau}}}
\newcommand{\Stressvarh}{\matb{\sigma}_h}

\newcommand{\Stressvarhtest}{\matb{\tau}_h}

\newcommand{\Sigmanormh}[1]{\left\| #1 \right\|_\Stressspaceh}

\newcommand{\Presspace}{{Q}}
\newcommand{\Presspaceh}{{\Presspace_h}}
\newcommand{\Presvar}{{p}}
\newcommand{\Presvartest}{{q}}
\newcommand{\Presvarh}{{p_h}}
\newcommand{\Presvarhtest}{{q_h}}

\newcommand{\Presnormh}[1]{|| #1 ||_\Presspaceh}

\newcommand{\Hcurldiv}[1]{H(\curl\divergence,#1)}
\newcommand{\Hone}{\vecb{H}^1}
\newcommand{\Poly}{{\mathbb{P}}}

\newcommand{\Forcevar}{{\vecb{f}}}

\newcommand{\cle}{\lesssim}
\newcommand{\RRR}{\mathbb{R}}
\newcommand{\om}{\Omega}
\newcommand{\const}{\mathrm{c}}

\newcommand{\Tref}{\widehat{T}}
\newcommand{\kernel}{K_h}

\newcommand{\proj}{\Pi}

\newcommand{\piola}{\mathcal{P}}
\newcommand{\covariant}{\mathcal{C}}
\newcommand{\covariantpiola}{ \mathcal{M}}
\newcommand{\MM}{\covariantpiola}

\newcommand{\dd}{\mathbb{D}}

\newcommand{\DD}{\mathcal{D}}
\let\d\partial
\newcommand{\grad}{\nabla}
\renewcommand{\div}{\operatorname{div}}
\newcommand{\dt}{{\tilde{d}}}
\newcommand{\ip}[1]{\langle {#1} \rangle}

\newenvironment{customlegend}[1][]{%
    \begingroup
    \csname pgfplots@init@cleared@structures\endcsname
    \pgfplotsset{#1}%
}{%
    \csname pgfplots@createlegend\endcsname
    \endgroup
}%

\def\addlegendimage{\csname pgfplots@addlegendimage\endcsname}

\pgfkeys{/pgfplots/number in legend/.style={%
        /pgfplots/legend image code/.code={%
            \node at (0.295,-0.0225){#1};
        },%
    },
}

\begin{document}

\title{A mass conserving mixed stress formulation for the Stokes equations}
\shorttitle{MCS formulation for the Stokes equations}

\author{%
{\sc Jay Gopalakrishnan}\thanks{Email: gjay@pdx.edu}\\[2pt]
Portland State University, PO Box 751, Portland OR 97207,USA\\[6pt]
{\sc and}\\[6pt]
{\sc
Philip L. Lederer\thanks{Corresponding author. Email: philip.lederer@tuwien.ac.at},
Joachim Sch{\"oberl}\thanks{Email: joachim.schoeberl@tuwien.ac.at}
} \\[2pt]
Institute for Analysis and Scientific Computing, TU Wien\\
Wiedner Hautstra{\ss}e 8-10, 1040 Wien, Austria
}
\shortauthorlist{Gopalakrishnan, Lederer, and Sch\"oberl}

\maketitle

\begin{abstract}
  {We propose a new discretization of a mixed stress formulation of
    the Stokes equations. The velocity $u$ is
    approximated with $H(\divergence)$-conforming finite elements
    providing exact mass conservation. While many standard methods use
    $H^1$-conforming spaces for the discrete velocity,
    $H(\divergence)$-conformity fits the considered variational
    formulation in this work. A new stress-like variable $\sigma$
    equalling the gradient of the velocity is set within a new
    function space~$H(\curl \divergence)$.  New matrix-valued finite
    elements having continuous ``normal-tangential'' components are
    constructed to approximate functions in $H(\curl \divergence)$.
    An error analysis concludes with optimal rates of convergence for
    errors in $u$ (measured in a discrete $H^1$-norm), errors in
    $\sigma$ (measured in $L^2$) and the pressure $p$ (also measured
    in $L^2$).
    The exact mass conservation property
    is directly related to another structure-preservation property 
    called {\it pressure robustness}, as shown
    by pressure-independent velocity error estimates.
    The computational cost measured in terms of interface degrees of
    freedom is comparable to old and new Stokes discretizations.       
  }

{mixed finite element methods; incompressible flows; Stokes equations}
\end{abstract}

\section{Introduction}
\label{sec:introduction}
We introduce a new method for the mixed stress formulation of the Stokes equations. Let $\Velvar$ and $\Presvar$ be the
velocity and pressure respectively.
Assume that we are given an external force $\Forcevar$, the kinematic
viscosity $\visc$ and a bounded domain $\Omega \subset \rr^d$ ($d=2$ or $3$)
with Lipschitz boundary $\d\Omega$. 
The  standard velocity-pressure formulation 
\begin{align} \label{eq::stokes}
  \left\{
  \begin{aligned}
  -\divergence(\visc \nabla \Velvar) + \nabla \Presvar & = \Forcevar && \textrm{in } \Omega,  \\
  \divergence (\Velvar) &=0 && \textrm{in } \Omega, \\
  \Velvar &= 0 &&\textrm{on } \partial \Omega,
\end{aligned}\right.
\end{align}
can be reformulated by introducing 
the variable $\Stressvar = \nu \nabla \Velvar$ as follows
\begin{align} \label{eq::mixedstressstokes-pre}
  \left\{
  \begin{aligned}
  \frac{1}{\nu}{\Stressvar} - \nabla \Velvar & = 0 &&\quad \textrm{in } \Omega,  \\
 \divergence(  \Stressvar) - \nabla \Presvar & = -\Forcevar &&\quad \textrm{in } \Omega,  \\
  \divergence (\Velvar) &=0 &&\quad \textrm{in } \Omega, \\
  \Velvar &= 0 &&\quad \textrm{on } \Gamma.
 \end{aligned}
\right.
\end{align}
Many authors have studied this formulation previously, e.g.,
\cite{MR1934446,MR1464150,MR1231323,Farhloulcanadian}. The initial
interest in this formulation as a numerical avenue appears to be due
to the fact that fluid stresses can be computed merely by algebraic
operations on $\sigma$ (i.e., no differentiation of computed variables
is needed).  In this paper, we 
 study the discretization errors and certain  interesting 
structure-preserving features of a new numerical method based
on~\eqref{eq::mixedstressstokes-pre}.

Although both formulations are formally equivalent, the mixed stress
formulation \eqref{eq::mixedstressstokes-pre} requires less regularity on the velocity field $\Velvar$. 
When considering a variational formulation of the classical
velocity-pressure formulation \eqref{eq::stokes}, the proper spaces
for the velocity and pressure are given by $\Hone_0(\Omega, \rr^d)$ and
$ L^2_0(\Omega)$, respectively. Here
$\Hone_0(\Omega,\rr^d)$ is the standard vector valued Sobolev space of
order one with zero boundary conditions and $L^2_0(\Omega)$ is the
space of square integrable functions with zero mean value. This pair
of spaces  fulfills the inf-sup condition or the LBB condition. Moreover,
the divergence operator from $H_0^1(\om, \RRR^d)$ to
$L_0^2(\om)$ is surjective.
Finite element discretizations of the velocity-pressure
formulation~\eqref{eq::stokes} 
is an active area of research~\cite{JLMNR:sirev}.
While many pairs of discrete velocity-pressure spaces are known to
satisfy the discrete LBB condition (needed to prove stability),
not all of them have the property that the divergence
operator from the discrete velocity space to the discrete pressure
space is surjective.
Methods that have this surjectivity property are 
particularly interesting because they provide numerical velocity
approximations that are exactly divergence free, leading to exact {\em
  mass conservation}.

Exact mass conservation (and consistency) further leads to a structure-preservation property
called {\em pressure robustness}. A feature of
solutions of~\eqref{eq::stokes} is that when the load $f$ changes
irrotationally (i.e., when $f$ is perturbed by a gradient field), then
the fluid velocity $u$ does not change (since the additional force can
be balanced solely by a pressure gradient).  Indeed, since
divergence-free functions are $L^2$-orthogonal to the irrotational
part of~$f$, and since the velocity $u$ is uniquely determined within
the divergence free subspace of $\Hone_0(\om, \RRR^d)$, the velocity
cannot be altered by irrotational changes in~$f$. This property is not
preserved by all finite element discretizations -- see
\cite{linke2014role} -- leading to velocity error estimates that
depend on the pressure approximation. A practical manifestation of
this is a phenomenon akin to ``locking,'' where the velocity error
increases as $\visc \to 0$ (even if the pressure error remains under
control). Methods that do not exhibit this limitation are called
pressure robust methods.  In the recent works of \cite{blms:2015,
  2016arXiv160903701L,Linke:2012,lmt:2016}, considering different
velocity and pressure spaces, it was shown that a (non-conforming)
modification of the load (right hand side) allows one to obtain
optimal pressure-independent velocity error estimates.

An alternative to this load modification approach is the use of finite
element spaces which lead to exactly divergence-free velocity
approximations. In this case, no load modification is needed and the
velocity error does not exhibit locking. A well-known example is the
$H^1$-conforming Scott-Vogelius element. However, it demands a special
barycentric triangulation of $\Omega$. Another approach, leading to
exactly divergence-free discretizations, is to abandon full
$H^1$-conformity and retain only the continuity of the normal
component of the velocity, i.e., use $H(\divergence)$-conforming
finite elements for approximating $u$ instead of $H^1$-conforming
finite elements. Such discretizations, tailored to approximate the
incompressibility constraint properly, were introduced by
\cite{cockburn2005locally, cockburn2007note} and for the Brinkman
Problem by \cite{stenberg2011}. Therein, and also in the work by
\cite{LS_CMAME_2016}, the $H^1$-conformity is treated in a weak sense
and a hybrid discontinuous Galerkin method is constructed. Their
choice of velocity and pressure space fulfills the discrete LBB
condition and moreover \cite{LedererSchoeberl2017} shows that it is
robust with respect to the polynomial order.

In this work, the idea of employing an $H(\divergence)$-conforming
velocity space is taken to an infinite dimensional variational setting
to obtain insights into possible spaces for $\sigma.$ Obviously such a
variational formulation cannot be derived using the standard
velocity-pressure formulation \eqref{eq::stokes} as it demands too
much regularity on the velocity. In contrast, the mixed stress
formulation \eqref{eq::mixedstressstokes-pre} is a perfect fit.  It leads
to a variational formulation requiring less regularity for $u$ and 
a new function space for $\sigma$, namely 
 $\Hcurldiv{\Omega}$. 
We call this formulation 
the {\em {\bf m}ass {\bf c}onserving mixed {\bf s}tress (MCS) formulation}.
To obtain a discretization, we design new non-conforming 
finite elements
for
 $\Hcurldiv{\Omega}$,  motivated by the TDNNS method for structural
mechanics introduced by~\cite{MR3712290, MR2826472, astriddiss}. 
Even though the resulting method, called the {\em MCS method,} 
includes the introduction of another variable, the computational costs
are comparable to other standard methods. In two dimensions, after a static
condensation step, where local element degrees of freedom are
eliminated, the approximation of the velocity with polynomials of
order $k$ requires $k+1$ coupling degrees of freedom on each element
interface for the $H(\divergence)$-conforming velocity space and $k$
for the stress space. This is the same number as for the reduced
stabilized (projected jumps) $H(\divergence)$-conforming hybrid
discontinuous Galerkin method introduced in \cite{LS_CMAME_2016}. By a
small modification, one could even reduce the coupling of the velocity
space by considering only relaxed $H(\divergence)$-conformity by the
same technique utilized in \cite{ledlehrschoe2017relaxedpartI, ledlehrschoe2018relaxedpartII}. Then the costs (for $k = 1$)
are the same as for the lowest order non-conforming $H^1$-based method. Similar cost comparisons can be made in 
three dimensions.

There appears to be multiple approaches for the analysis of our new
scheme. In this paper, we focus on one of these possible approaches,
which uses a discrete $H^1$-like norm for $u$ and a $L^2$ norm
for $\sigma$. Even though $u$ is approximated using
$H(\divergence)$-conforming elements, the use of the discrete
$H^1$-like norm for velocity errors permits easy comparison with the
classical velocity-pressure formulation.  An analysis in more
``natural'' norms (i.e., the $H(\divergence)$-norm for $u$ and
$\Hcurldiv{\Omega}$-norm for $\sigma$) is the topic of a forthcoming work.


The paper is organized as follows. We begin with Section
\ref{sec:notation} where we define the notations and prove certain 
preliminary results that we shall use throughout this work.
In Section \ref{sec:derivation} we present the  
new MCS variational formulation of the Stokes problem.
Section~\ref{sec:discretevarform} defines the discrete variational
formulation and the MCS method. 
After revealing the continuity requirements across element interfaces 
necessary for being conforming in~$\Hcurldiv{\Omega}$, 
we then define new non-conforming finite elements for the $\sigma$ variable 
in Section~\ref{sec::finiteelements}.
All technical details
needed to prove stability in certain discrete norms
and convergence of the new method 
are included in Section~\ref{sec::apriorianalysis}.
In Section~\ref{sec:numerics} we present various numerical examples to 
illustrate  the theory.

\section{Preliminaries}   \label{sec:notation}

In this section we define the notations we use throughout 
and establish properties of certain 
Sobolev spaces we shall need later.

Let $\Omega \subset \rr^d$
be an open bounded domain 
with Lipschitz boundary $\Gamma:= \partial \Omega$. 
Throughout, $d$ is either $2$ or $3$. Let $\DD(\om)$ or $\DD(\om, \RRR)$
 denote the set of infinitely 
differentiable compactly supported real-valued functions on $\om$ and let $\DD'(\om)$ denote the space of distributions as usual. To indicate 
vector and matrix-valued functions on $\om$, we include the range in the notation:
 $\DD(\om, \RRR^d) = \{ u: \om \to \RRR^d| \; u_i \in \DD(\om)\}$. Such notations are 
extended in an obvious fashion to other function spaces as needed. E.g., while $L^2(\om) = L^2(\om, \RRR)$ denotes the space of square integrable real-valued functions on $\om$, analogous vector and matrix-valued function spaces are defined by
\begin{align*}
L^2(\om, \rr^d) := \left\{ u : \om \to \rr^d \big| u_i \in L^2(\om)\right\} \quad \textrm{and} \quad L^2(\om, \rr^{d\times d}) := \left\{ \sigma : \om \to \rr^{d\times d} \big| \sigma_{ij} \in L^2(\om)\right\}.
\end{align*}
Similarly, $\DD'(\om, \RRR^d)$ denotes the space of distributions whose
components are distributions in $\DD'(\om)$, $H^m(\Omega, \rr^{d\times d})$, denotes
the space of matrix-valued functions whose entries are in the standard
Sobolev
space $H^m(\om)$ for any $m\in \RRR$, etc.

Certain differential operators have different definitions depending on
context. By ``curl'' we mean any of the following three differential
operators
\begin{align*}
  \curl(\phi)
  & = (-\partial_2 \phi, \partial_1 \phi)^\trans, 
  && \text{ for } \phi \in \DD'(\om, \RRR) \text{ and } d=2,
  \\
  \curl( \phi)
  & = -\partial_2 \phi_1+ \partial_1 \phi_2,
  && \text{ for } \phi \in \DD'(\om, \RRR^2) \text{ and } d=2,
  \\
  \curl (\phi)
  & 
    = 
    (\d_2\phi_3 - \d_3 \phi_2, \d_3\phi_1 - \d_1\phi_3, 
    \d_1\phi_2 - \d_2\phi_1)^\trans
  && \text{ for } \phi \in \DD'(\om, \RRR^3) \text{ and } d=3,     
\end{align*}
where $(\cdot)^\trans$ denotes the transpose and $\d_i$ abbreviates
$\d/\d x_i$.  The type of the operand determines which operator
definition to apply in any context, so there will be no confusion.
Similarly, $\grad$ is to be understood from context as an operator
that results in either a vector whose components are
$[\grad \phi]_i = \d_i \phi$ for $\phi \in \DD'(\om, \RRR)$ or a
matrix whose entries are $[\grad \phi]_{ij} = \d_j \phi_i$ for
$\phi \in \DD'(\om, \RRR^d)$. Finally, in a similar manner, we
understand $\div(\phi)$ as either $\sum_{i=1}^d \d_i \phi_i$ for
vector-valued $\phi \in \DD'(\om, \RRR^d),$ or the row-wise divergence
$\sum_{j=1}^d \d_j \phi_{ij}$ for matrix-valued
$\phi \in \DD'(\om, \RRR^{d \times d})$.

Let $\dt= d(d-1)/2$ (so that $\dt=1$ and $3$ for $d=2$ and 3, respectively).
The following Sobolev spaces for $d=2,3$ are essential in our study:
\begin{align*}
  \vecb{H}(\divergence,\Omega)
  &= \{ \Velvar \in L^2(\Omega, \rr^d): \divergence(\Velvar) \in
    L^2(\Omega) \},
  \\
  \vecb{H}(\curl,\Omega) 
  &= \{ \Velvar \in L^2(\Omega, \rr^d):
    \curl(\Velvar) \in L^2(\Omega,
    \rr^{\dt}) \},
\\
  H^{-1}(\curl, \Omega)  
  &= \{ \vecb{\phi} \in H^{-1}(\Omega, \rr^d): \curl(\vecb{\phi}) \in
    H^{-1}(\Omega, \rr^{\dt})\}, 
\\
  \Hcurldiv{\Omega} 
  & = \{ \Stressvar \in L^2(\Omega, \rr^{d \times d}):
    \curl(\divergence (\Stressvar)) \in H^{-1}(\Omega, \rr^{\dt}) \}.
\end{align*}
A well-known trace theorem permits us to define
$H_0(\div,\om) = \{ u \in H(\div, \om): u \cdot n|_\Gamma =0\}.$ Here,
$n$ denotes the outward unit normal on $\Gamma$. In other occurrences,
it may denote the unit outward normal on boundaries of other domains
determined from context.

The action of a continuous linear functional $f$ on an element $x$ of
a topological space $X$ is denoted by $\ip{ f, x}_X$, e.g., the action
of a distribution $F \in \DD'(\om, \RRR^d)$ on a
$\phi \in \DD(\om, \RRR^d)$ is denoted by
$\ip{F, \phi}_{\DD(\om, \RRR^d)}$.  We omit the subscript in
$\ip{\cdot, \cdot}$ when its obvious from context.  When $X$ is a
Hilbert space, we use $X^*$ to denote its dual space. Recall that
$H_0^1(\om)^* = H^{-1}(\om)$. Note that any $f \in
H^{-1}(\om)$ is a distribution and 
\begin{equation}
  \label{eq:4}
  \ip{ f, \phi}_{H_0^1(\om)} = \ip{ f, \phi}_{\DD(\om)}  
\end{equation}
for all $\phi \in \DD(\om)$. The inner product of $X$ is denoted by
$(\cdot, \cdot)_X$. When $X$ is $L^2(\om), L^2(\om, \RRR^d)$, or
$L^2(\om, \RRR^{d \times d})$, we abbreviate $(\cdot, \cdot)_X$ to
simply $(\cdot, \cdot)$.

\begin{lemma}
  \label{lem:Hodiv_subset_Hmcurl}
  If $F \in H_0(\div, \om)^*$, then $F$ is in $H^{-1}(\curl, \om)$ and 
  for all $v \in H_0^1(\om)$, 
  \[
  \ip{ \curl(F), v}_{H_0^1(\om)} 
  = \ip{F, \curl(v)}_{H_0(\div, \om)}.
  \]
\end{lemma}
\begin{proof}
  For any $F \in H_0(\div, \om)^*$, 
  by the Reisz representation theorem, there exists a $q^F \in
  H_0(\div, \om)$ satisfying 
  \begin{equation}
    \label{eq:5}
    \ip{ F, v}_{H_0(\div, \om)} = (q^F, v) + (\div(q^F), \div(v)).
  \end{equation}
  for $v \in H_0(\div, \om)$.  Choosing $v \in \DD(\om, \RRR^d)$ we
  conclude that $F$ is the distribution
  $ F = q^F - \grad \div(q^F) \in H^{-1}(\om, \RRR^d).$ This implies
  that $\curl(F) = \curl (q^F) \in H^{-1}(\om, \RRR^{\dt})$.  Thus
  $F \in H^{-1}(\curl, \om)$.

  Moreover, for all $\phi \in \DD(\om,\RRR^d)$, using~\eqref{eq:4},
  \begin{align*}
    \ip{ \curl(F), \phi}_{H_0^1(\om, \RRR^d)}
    & = \ip{ \curl(q^F), \phi}_{H_0^1(\om, \RRR^d)} 
      = \ip{ \curl(q^F), \phi}_{\DD(\om, \RRR^d)}
      = (q^F, \curl(\phi)).
  \end{align*}
  By the density of $\DD(\om, \RRR^d)$ in $H_0^1(\om, \RRR^d)$, we
  obtain 
  \[
  \ip{ \curl(F), v}_{H_0^1(\om, \RRR^d)}
  = (q^F, \curl(v))
  \]
  for all $v \in H_0^1(\om, \RRR^d)$. The proof is now complete due
  to~\eqref{eq:5}.
\end{proof}

In the proof of the next result, we use a
``regular decomposition'' of $H_0(\div, \om)$. Namely, there exists a
$C>0$ such that given any $v \in H_0(\div, \om)$, there is a 
$\phi_v \in H_0^1(\om, \RRR^{\dt})$ and a $z_v \in H_0^1(\om, \RRR^d)$ such
that
\begin{equation}
  \label{eq:1}
v = \curl(\phi_v) + z_v, \qquad 
\| \phi_v\|_{H^1(\om, \rr^\dt)}  +  \| z_v\|_{H^1(\om, \RRR^d)} 
\le C \| v \|_{H(\div, \om)}.  
\end{equation}
Many authors have stated this decomposition under various assumptions on
$\om$. Since there are too many to list here, we content ourselves by
pointing to~\cite[Lemma~5]{DemloHiran14} where one can find the result
under the current assumptions on $\om$ and further references.

\begin{theorem} \label{thm:eqdualspace} 
  The equality 
  \[
  H_0(\divergence, \Omega)^* = H^{-1}(\curl, \Omega)
  \]
  holds algebraically and topologically.
\end{theorem}
\begin{proof}
  Lemma~\ref{lem:Hodiv_subset_Hmcurl} shows that
  $H_0(\div, \om)^* \subseteq H^{-1}(\curl, \om).$ 
  To show  
  $H^{-1}(\curl,\om) \subseteq H_0(\div, \om)^*$, 
  let
  $g \in H^{-1}(\curl,\om)$. Using the decomposition~\eqref{eq:1},
  set 
      \begin{equation}
        \label{eq:3}
      \ip{G, v}_{H_0(\div,\om)} := \ip{\curl(g), \phi_v}_{H_0^1(\om,
        \RRR^\dt)}
      + \ip{ g, z_v}_{H_0^1(\om, \RRR^d)}.        
      \end{equation}
    Due to the stability estimate of~\eqref{eq:1}, $G$ is a continuous
    linear functional in $H_0(\div, \om)^*$. 
    By Lemma~\ref{lem:Hodiv_subset_Hmcurl}, $G$ is $H^{-1}(\curl,
    \om)$.
    It suffices to show $G$ coincides with $g$ (as an element of $H^{-1}(\om, \RRR^d)$).  To this end, let
    $w \in H_0^1(\om, \RRR^d)$. Since
    $H_0^1(\om, \RRR^d) \hookrightarrow H_0(\div, \om),$ we have
    $\ip{ G, w}_{H_0^1(\om, \RRR^d)} = \ip{G, w}_{H_0(\div, \om)}$, so using decomposition \eqref{eq:1}
    \[
    \ip{G, w}_{H_0^1(\om, \RRR^d)}
    = \ip{ \curl(g), \phi_w}_{H_0^1(\om,\RRR^{\dt})}
    +
    \ip{ g, z_w}_{H_0^1(\om, \RRR^d)}.
    \]
    Since both $w$ and $z_w$ are in $H_0^1(\om, \RRR^d)$ the equality
    $w = \curl (\phi_w) + z_w$ implies that
    $\curl (\phi_w) \in H_0^1(\om, \RRR^d)$. 

    Let $\varphi_n \in\DD(\om, \RRR^\dt)$ converge to  $\phi_w$ in 
    $H_0^1(\om, \RRR^\dt)$.
    Using~\eqref{eq:4}, 
    \[
    \ip{ \curl(g), \varphi_n}_{H_0^1(\om, \RRR^\dt)} = 
    \ip{ \curl(g), \varphi_n}_{\DD(\om, \RRR^\dt)} = 
    \ip{ g, \curl(\varphi_n)}_{\DD(\om, \RRR^d)} = 
    \ip{ g, \curl(\varphi_n)}_{H_0^1(\om, \RRR^d)}.
    \]
    Since $\curl(g)$ is in $H^{-1}(\om, \RRR^d),$ 
    the left-most term converges to
    $ \ip{ \curl(g), \phi_w}_{H_0^1(\om, \RRR^d)}.$
    The right-most term
    $    \ip{ g, \curl(\varphi_n)}_{H_0^1(\om, \RRR^d)}$ must converge to
    the same limit and  since $ \curl (\phi_w)$ is in $H_0^1(\om, %
    \RRR^d)$, the limit must equal 
    $\ip{ g, \curl(\phi_w)}_{H_0^1(\om, \RRR^d)}.$ Thus
    $
    \ip{ \curl(g), \phi_w}_{H_0^1(\om, \RRR^d)} = 
    \ip{ g, \curl(\phi_w)}_{H_0^1(\om, \RRR^d)}
    $
    and consequently, 
    $
    \ip{G, w}_{H_0^1(\om, \RRR^d)}
    = \ip{ g, \curl(\phi_w) + z_w}_{H_0^1(\om, \RRR^d)}
    = \ip{g, w}_{H_0^1(\om, \RRR^d)}.
    $
    This proves that $G = g,$ so $g \in H_0(\div, \om)^*$.

  Finally, the stated topological equality follows if we show that
  $\| f \|_{H_0(\div,\om)^*} \sim \| f\|_{H^{-1}(\curl,\om)}$, where
  ``$\sim$'' denotes norm equivalence. Note that by~\eqref{eq:1} and triangle
  inequality,
  $
    \| \phi_v\|_{H^1(\om, \rr^\dt)}  +  \| z_v\|_{H^1(\om, \RRR^d)} 
   $ 
   $\sim$
   $
    \| v \|_{H(\div,\om)}.
  $
  For any $f \in H_0(\div,\om)^*$,
  \begin{align*}
    \| f \|_{H_0(\div,\om)^*}
    & = \sup_{v \in H_0(\div,\om)} 
    \frac{ \ip{ f, v}_{H_0(\div,\om)}}{\| v\|_{H(\div,\om)}}
      \\
    & \sim \;
    \sup_{\phi \in H_0^1(\om, \RRR^{\dt}), \; z \in H_0^1(\om, \RRR^d)}
    \frac{ \ip{ f, \curl(\phi) + z}_{H_0(\div,\om)}}
    { \| \phi\|_{H^1(\om, \rr^\dt)}  +  \| z\|_{H^1(\om, \RRR^d)}
    }    
    && \text{by \eqref{eq:1}}
    \\
    & = 
    \sup_{\phi \in H_0^1(\om, \RRR^{\dt}), \; z \in H_0^1(\om, \RRR^d)}
    \frac{ \ip{ \curl(f), \phi}_{H_0^1(\om, \RRR^\dt)}
      + \ip{f, z}_{H_0^1(\om, \RRR^d)}}
      { \| \phi\|_{H^1(\om, \rr^\dt)}  +  \| z\|_{H^1(\om,
      \RRR^d)} }    
      && \text{by Lemma~\ref{lem:Hodiv_subset_Hmcurl}}
    \\
    & \sim\;  \| f \|_{H^{-1}(\om, \RRR^d)} + \| \curl(f) \| _{H^{-1}(\om, \RRR^d)}.
  \end{align*}
  Thus the $H_0(\div,\om)^*$ and $H^{-1}(\curl,\om)$ norms are equivalent.
\end{proof}

\section{Derivation of the MCS formulation of the Stokes equations} \label{sec:derivation}

The goal of this section is to quickly derive  
a variational formulation of the mixed
stress formulation of the Stokes system
\eqref{eq::mixedstressstokes-pre}.
Using the trace of a matrix $\trace{\Stressvartest} := \sum_{i=1}^d \Stressvartest_{ii}$ we define the 
 deviatoric part by
 \[
\Dev{\tau} = \tau - \frac{\trace{\tau}}{d} \id,
\]
where $\id$ denotes the identity matrix.
Observe that due to 
$\divergence(u) = 0$, we have
\begin{align} \label{eq::devsigma}
\Dev{\Stressvar} = \Dev{\visc \nabla \Velvar} =  \visc\nabla \Velvar - \frac{\visc}{d} \trace{\nabla \Velvar} \id = \visc (\nabla \Velvar - \frac{1}{d} \divergence(\Velvar) \id) =  \visc \nabla \Velvar.
\end{align}
Thus $\Stressvar = \visc \nabla \Velvar$ in \eqref{eq::mixedstressstokes-pre} 
only represents the deviatoric part of the velocity gradient. Hence we
revise \eqref{eq::mixedstressstokes-pre}  to 
\begin{subequations}
  \label{eq::mixedstressstokes}
  \begin{alignat}{2}
  \label{eq::mixedstressstokes-a}
  \frac{1}{\nu}{\Dev\Stressvar} - \nabla \Velvar & = 0 \quad && \textrm{in } \Omega,  \\
  \label{eq::mixedstressstokes-b}
 \divergence(  \Stressvar) - \nabla \Presvar & = -\Forcevar \quad && \textrm{in } \Omega,  \\
  \label{eq::mixedstressstokes-c}
  \divergence (\Velvar) &=0 \quad&& \textrm{in } \Omega, \\
  \label{eq::mixedstressstokes-d}
  \Velvar &= 0 \quad && \textrm{on } \Gamma.
 \end{alignat}
\end{subequations}
We proceed to develop a variational formulation
for~\eqref{eq::mixedstressstokes}.

For the reasons described in the introduction, 
we want to derive a weak  formulation  where 
the velocity $u$ and the pressure $p$ 
belong respectively to the following spaces.
\begin{align*}
  \Velspace &:= \vecb{H}_0(\divergence,\Omega) = \{ \Velvar \in
              \vecb{H}(\divergence,\Omega): \; \Velvar \cdot \normal = 0 \textrm{ on } \Gamma \}, \\
  \Presspace &:= L^2_0(\Omega):=\{ \Presvartest \in L^2(\Omega): \; \int_\Omega \Presvartest \dx =0 \}.
\end{align*}
We begin 
with~\eqref{eq::mixedstressstokes-c}. 
Multiplying 
\eqref{eq::mixedstressstokes-c} with a test function $\Presvartest \in
\Presspace$ and integrating over the domain $\Omega$, we obtain the 
familiar equation
\begin{align}
  \label{eq:8}
  (\divergence(\Velvar), \Presvartest) =0.
\end{align}
Proceeding next to~\eqref{eq::mixedstressstokes-b}, which must
be tested with a $v \in V$, we see that $\sigma$ in addition to
being in $L^2(\om, \RRR^{d\times d})$, must also be such 
that $\divergence(\sigma)$ can continuously ``act'' on $v$, i.e.,
$\div(\sigma) \in 
H(\divergence, \Omega)^*$. By Theorem~\ref{thm:eqdualspace}, this is
the same as requiring that 
\begin{equation}
  \label{eq:2}
\div(\sigma) \in H^{-1}(\curl, \om).  
\end{equation}
Since any $\sigma$ in $L^2(\om, \RRR^{d\times d})$ has
$\divergence(\Stressvar) \in H^{-1}(\Omega, \rr^d)$, the non-redundant
requirement that emerges from~\eqref{eq:2} is that
$
\curl(\divergence(\Stressvar)) \in  H^{-1}(\Omega, \rr^{\dt}).
$
This leads to the definition
\begin{align*}
  \Stressspace = 
  \{ \tau \in \Hcurldiv{\Omega}
  :\; \trace{\tau}=0\}
\end{align*}
where the requirement $\trace\tau=0$ is motivated 
by \eqref{eq::devsigma}. Thus, testing~\eqref{eq::mixedstressstokes-b}
with a $v \in H_0(\div, \om)^*$ and integrating the pressure term by parts,   we have
\begin{align}
  \label{eq:9}
  \langle \divergence( \Stressvar),\Velvartest
  \rangle_{H_0(\divergence,\Omega)}
 + (\divergence(\Velvartest), \Presvar) = 0.
\end{align}
Finally, we multiply~\eqref{eq::mixedstressstokes-a} with 
a test function $\tau \in \Stressspace$ to obtain $(\visc^{-1}
\Dev\sigma, \tau) - (\nabla u, \tau) = 0$. Since 
\begin{equation}
  \label{eq:6}
  (\tau, \nabla v) = -\ip{ \div(\tau), v}_{H_0(\div, \om)}, \qquad
  \text{ for all }\tau \in \Sigma,\;  v \in H_0^1(\om , \RRR^d), 
\end{equation}
using the fact that the exact velocity is in $H_0^1(\om, \RRR^d)$, we obtain  
\begin{equation}
  \label{eq:7}
  (\visc^{-1} \Dev \sigma, \Dev \tau) +
  \ip{ \div(\tau), u}_{H_0(\div, \om)} = 0.
\end{equation}
Note that in this derivation, 
while the normal trace of the velocity is an essential boundary
condition included in the space $\Velspace$, 
the zero  tangential
velocity boundary conditions was incorporated weakly 
as a natural boundary condition in~\eqref{eq:7}.

Collecting~\eqref{eq:7},~\eqref{eq:9} and~\eqref{eq:8}, we summarize
the derived  weak formulation: given $f \in H_0(\div, \om)^*$, find $(\Stressvar, \Velvar, \Presvar) \in \Stressspace \times \Velspace \times \Presspace$ such that
\begin{align} \label{eq::mixedstressstokesweak}
    \left\{
  \begin{aligned}
    (\nu^{-1} \Dev{\Stressvar}, \Dev{\Stressvartest} ) 
    + 
    \ip{ \divergence(\Stressvartest),  \Velvar}_{H_0(\div, \om)} 
    & = 0 
    &&\text{ for all } \Stressvartest \in \Stressspace, 
    \\
    \ip{\divergence(\Stressvar), \Velvartest}_{H_0(\div, \om)}
    +
    (\divergence(\Velvartest), \Presvar) 
    & = 
    -\ip{\Forcevar, \Velvartest}_{H_0(\div, \om)} 
    &&\text{ for all } \Velvartest \in \Velspace,
    \\
    (\divergence(\Velvar), \Presvartest) &=0 
    &&\text{ for all } \Presvar \in \Presspace.
\end{aligned}
\right.                                                         
\end{align}
In the remainder of the paper, we present an approximation of the weak
formulation \eqref{eq::mixedstressstokesweak}. Its possible to prove
that~\eqref{eq::mixedstressstokesweak} is well posed. However, since we
shall focus on a discrete analysis of a nonconforming scheme based
on~\eqref{eq::mixedstressstokesweak}, we shall not make direct use of
the wellposedness in this work. As a final remark
on~\eqref{eq::mixedstressstokesweak}, 
  we note that functions in $\Stressspace$ equal its  deviatoric. 
Thus we could remove ``dev'' in the first term
  of~\eqref{eq::mixedstressstokesweak}. However, we keep it to remind
  ourselves that $\sigma$ only approximates 
  the deviatoric part of $\visc \nabla u.$

\begin{remark}[Boundary conditions] \label{rem::boundaryconditionsinf}
  In this work we only consider homogeneous Dirichlet boundary
  conditions of the velocity, $\Velvar = 0$ on $\Gamma$. However, also other types of  boundary conditions as for example slip boundary conditions for the velocity and homogeneous Neumann boundary conditions $(-\visc \nabla \Velvar + p \id) \cdot \normal  = (-\sigma + p \id) \cdot \normal = 0$ are possible. A detailed analysis regarding this topic is included in a forthcoming work.

\end{remark}

\section{A discrete formulation} \label{sec:discretevarform}

We present the discrete MCS method in this section. It is a
non-conforming method based on the MCS weak
formulation~\eqref{eq::mixedstressstokesweak}. We shall begin by
understanding the conformity requirements of $\Hcurldiv{\Omega}$ and
then present the method.

Suppose  $\Omega$ is partitioned by a  shape regular and quasiuniform triangulation $\mesh$ consisting of triangles and tetrahedrons in two and three dimensions, respectively. 
Here $h $ denotes the maximum of the diameters of all elements in $\mesh$. 
Due to quasiuniformity  $h \approx \textrm{diam}(T)$ for any $ T \in \mesh$. 
The set of element interfaces and boundaries is denoted by $\facets$. This set is further split into facets on the domain boundary $F \subset \facets \cap \Gamma =:\facetsext$ and facets in the interior $F \subset \facets \cap \Omega =:\facetsint$. There holds $\facets = \facetsint \cup \facetsext$. On each facet $F \in \facetsint$ we denote by $\jump{\cdot}$ the usual jump operator. For facets on the boundary the jump operator is just the identity. 
On each element boundary, and similarly 
on each facet on the global boundary, 
using the outward unit normal vector  $\normal$,
the normal and tangential trace of a 
 smooth enough $\Velvar : \om \to \RRR^d$ is  defined by 
\begin{align*}
  \Velvar_\normal = \Velvar \cdot \normal \quad \textrm{and} \quad \Velvar_\tangential = \Velvar - \Velvar_\normal \normal.
\end{align*}
According to this definition the normal trace is a scalar function and the tangential
trace is a vector function. In two dimensions, we may fix the symbol $t$ to
a unit tangent vector, obtained say
by rotating $n$ anti-clockwise by 90 degrees (thus $t = n^\perp$), so that 
 $\Velvar_\tangential = (\Velvar \cdot \tangential) \tangential$.  In
 a similar manner 
for a smooth enough  $\Stressvar: \om \to
 \RRR^{d \times d}$ we set
\begin{align*}
  \Stressvar_{\normal\normal} =   \Stressvar : (\normal \otimes \normal) = \normal^\trans \Stressvar \normal \quad \textrm{and} \quad \Stressvar_{\normal\tangential} =  \Stressvar \normal -  \Stressvar_{\normal\normal} \normal.
\end{align*}
Thus we have a scalar ``normal-normal component'' and a vector-valued
``normal-tangential component,'' and in two dimensions $t$ may be
thought of as  a unit tangent vector and 
$\Stressvar_{\normal\tangential} = (\tangential^\trans\Stressvar
\normal)\tangential$.

The next result shows the conformity requirements 
in $\Hcurldiv{\Omega}$. Just as continuity of the normal component 
across element interfaces is needed for 
$H(\divergence, \om)$-conformity, we shall see 
that {\em continuity of the 
normal-tangential component of tensors}
is needed for $\Hcurldiv{\Omega}$-conformity.
Let
\begin{align*}
H^m(\mesh):= \{ v \in L^2(\Omega): v|_T \in H^m(T) \text{ for all } T \in \mesh\}.
\end{align*}
For $\omega \subset \om$ we use $(\cdot, \cdot)_\omega$ to denote the inner product of $L^2(\omega),
L^2(\omega, \RRR^d),$ or $L^2(\omega, \RRR^{d \times d})$ and similarly also $|| \cdot ||^2_\omega := ( \cdot,\cdot)_\omega$.

\begin{theorem} \label{th::normtangcont}
  Suppose  $\Stressvar$ is in $H^1(\mesh, \rr^{d \times d})$ and
  $\Stressvar_{\normal\normal}|_{\d T} \in H^{1/2}(\partial T)$ for
  all elements $T \in \mesh$. Assume that the normal-tangential
  trace $\Stressvar_{nt}$ is continuous across element
  interfaces. Then $\sigma$ is in $\Hcurldiv{\Omega}$ and moreover
  \begin{equation}
    \label{eq:10}
    \ip{ \div(\sigma), v}_{H_0(\div, \om)}
    = 
    \sum_{T \in \mesh}
    \left[
      (\div(\sigma), v)_T - \ip{ v_n, \sigma_{nn}}_{H^{1/2}(\d T)}
    \right]
  \end{equation}
  for all $v \in H_0(\div, \om).$
\end{theorem}
  \begin{proof}
    Using the definition of the distributional divergence and integration by parts yields
    \begin{align*}
\langle \divergence (\Stressvar), \phi \rangle = - \int_\om \Stressvar : \nabla \phi \dx = \sum\limits_{T \in \mesh} \int_T \divergence(\Stressvar) \cdot \phi \dx - \int_{\partial T} \Stressvar_\normal \cdot \phi \ds
    \end{align*}
     for any $\phi \in \mathcal{D}(\om, \rr^d)$.
    Splitting the boundary term into a tangential and a normal part we
    obtain
    \begin{align*}
\sum\limits_{T \in \mesh} - \int_{\partial T} \Stressvar_\normal \cdot
      \phi \ds &= \sum\limits_{T \in \mesh} - \int_{\partial T}
                 \Stressvar_{\normal\normal}  \phi_\normal \ds -
                 \int_{\partial T} \Stressvar_{\normal\tangential}
                 \cdot  \phi_\tangential \ds \\ 
               &= \sum\limits_{T \in \mesh} - \int_{\partial T} \Stressvar_{\normal\normal}  \phi_\normal \ds - \sum\limits_{F \in \facets} \int_{F} \jump{\Stressvar_{\normal\tangential}}  \cdot \phi_\tangential \ds.
    \end{align*}
    As $\Stressvar_{nt}$ is continuous across element interfaces, the
    second term vanishes. Hence 
    \begin{align}
      \label{eq:11}
      \langle \divergence (\Stressvar), \phi \rangle &= \sum\limits_{T \in \mesh} \int_T \divergence(\Stressvar) \cdot \phi \dx - \int_{\partial T} \Stressvar_{\normal\normal}  \phi_\normal \ds \\\nonumber
                                                     &\le \sum\limits_{T \in \mesh} ||\divergence(\Stressvar)||_{T} ||\phi||_{T} + || \Stressvar_{\normal\normal}||_{H^{1/2}(\partial T)} ||\phi_\normal||_{H^{-1/2}(\partial T)} \le \const(\sigma) ||\phi||_{H(\divergence,\om)},
    \end{align}
    where $\const(\Stressvar)$ is a constant depending on
    $\Stressvar$. Since $ \mathcal{D}(\om, \rr^d)$ is dense in
    $H_0(\div, \om)$, we conclude that $\div(\sigma)$ is in
    $H_0(\div, \om)^*$. Hence by Theorem~\ref{thm:eqdualspace},
    $\sigma \in \Hcurldiv{\Omega}$. The identity~\eqref{eq:10} also follows
    from~\eqref{eq:11} and  a density argument.
  \end{proof}

According to Theorem~\ref{th::normtangcont} one of the  sufficient
conditions for conformity in  $\Hcurldiv{\Omega}$ is normal-tangential
continuity.  Full conformity is obtained under the further
condition that  $\Stressvar_{\normal\normal} \in H^{1/2}(\partial T),$
which demands more continuity:  if the normal-normal component trace is 
continuous at vertices and edges in two and three dimensions,
respectively, then the $\sigma$ considered in Theorem~\ref{th::normtangcont}
would satisfy $\Stressvar_{\normal\normal} \in H^{1/2}(\partial
T).$ If this latter  constraint is relaxed, much simpler
elements can be constructed, as we shall see in
Section~\ref{sec::finiteelements}.

Theorem~\ref{th::normtangcont} provides the motivation for the
definition of the discrete space $\Sigma_h$ below, even though
$\Sigma_h \not\subset \Sigma.$ Let $\Poly^k(T)$ denote the space of
polynomials of degree at most~$k$ restricted to $T$. Let $\Poly^k(T,
\RRR^d)$ and $\Poly^k(T, \RRR^{d \times d})$ 
denote the space of vector and matrix-valued functions on $T$ whose
components are in $\Poly^k(T)$, and let 
\[
\Poly^k(\mesh) = \prod_{T \in \mesh} \Poly^k(T), \qquad 
\Poly^k(\mesh, \RRR^d) = \prod_{T \in \mesh} \Poly^k(T, \RRR^d),
\qquad
\Poly^k(\mesh, \RRR^{d \times d}) = \prod_{T \in \mesh} \Poly^k(T,
\RRR^{d \times d}).
\]
Define 
\begin{align}
  \label{eq:13}
  \Stressspaceh &:= \{ \tau_h \in \Poly^k(\mesh,\rr^{d \times
                  d}):\;\trace{\tau_h} = 0,\;
                  \jump{(\tau_h)_{\normal\tangential}} =0,\;
                  (\tau_h)_{\normal\tangential} \in \Poly^{k-1}(F,
                  \rr^{d-1}) \text{ for all } F \in \facets\} 
\\
  \Velspaceh &:= \Poly^k(\mesh,\rr^{d}) \cap V,\\
  \Presspaceh &:= \Poly^{k-1}(\mesh) \cap Q.
\end{align}
Note that the normal-tangential component $(\tau_h)_{nt}|_F$ of any
$\tau_h \in \Poly^k(\mesh,\rr^{d \times d})$ is a tangential vector
field whose values are in the tangent plane parallel to the facet $F$.  By a
slight abuse of notation, we do not distinguish between this tangent
plane and the isomorphic $\RRR^{d-1}$ (when we write  statements like 
``$\tau_{\normal\tangential} \in \Poly^{k-1}(F, \rr^{d-1})$'' above 
in~\eqref{eq:13}).

For the derivation of a discrete variational formulation with these
spaces, we return to~\eqref{eq::mixedstressstokesweak} and identify
these bilinear forms:
\begin{align*}
  & \ablf: L^2(\om, \RRR^{d \times d}) \times L^2(\om, \RRR^{d \times
    d})  \rightarrow \rr,
  &&  \blfone: \Velspace \times \Presspace \rightarrow \rr,
  \\
  &\ablf(\Stressvar, \Stressvartest) := 
(\nu^{-1}     \Dev{\Stressvar}, \Dev{\Stressvartest}),    
  &&\blfone(\Velvar, \Presvar) := 
     (\divergence(\Velvar), \Presvar).
\end{align*}
To handle  the terms with the divergence of stress variables, we 
define another bilinear form 
\[
b_2 : \{ \tau \in H^1(\mesh, \RRR^{d \times d}): \;
\jump{\tau_{nt}} =0\}  \times 
\{ v \in  H^1(\mesh, \RRR^d): \; \jump{v_n} =0\}  \to \rr
\]
motivated 
by the  identity~\eqref{eq:10} of Theorem \ref{th::normtangcont}:
\begin{align}
\blftwo(\tau, v) &:=  \sum\limits_{T \in \mesh} \int_T
                   \divergence(\tau) \cdot v \dx - \sum\limits_{F \in
                   \facets} \int_F \jump{\tau_{nn}} v_n \ds. 
\label{eq::blftwoequione} 
\end{align}
By  integration by parts, we find the equivalent representation
\begin{align}
  \blftwo(\tau, v)
    &=  -\sum\limits_{T \in \mesh} \int_T \tau : \nabla v \dx + \sum\limits_{F \in \facets} \int_F \tau_{nt} \cdot \jump{v_t} \ds \label{eq::blftwoequitwo}
\end{align}
since $\jump{\tau_{nt}} =0$ and $\jump{v_n} =0$.
When trial and test functions are in the domain of these forms, the
MCS weak form~\eqref{eq::mixedstressstokesweak} can be rewritten in
terms of these forms. 

The discrete  MCS method finds 
$(\Stressvarh, \Velvarh, \Presvarh) \in \Stressspaceh \times \Velspaceh \times \Presspaceh$ satisfying
\begin{align} \label{eq::discrmixedstressstokesweak}
    \left\{
  \begin{aligned}
 \ablf (\Stressvarh ,\Stressvarhtest) + \blftwo(\Stressvarhtest, \Velvarh) & = 0 &&\text{ for all } \Stressvarhtest \in \Stressspaceh,  \\
\blftwo(\Stressvarh, \Velvarhtest) + \blfone(\Velvarhtest, \Presvarh) & = (-\Forcevar, \Velvarhtest)  &&\text{ for all } \Velvarhtest \in \Velspaceh,\\
  \blfone(\Velvarh, \Presvarhtest) &=0 &&\text{ for all } \Presvarhtest \in \Presspaceh.
\end{aligned}
\right.               \tag{MCS}                                                
\end{align} 
Note that the velocity space is the well known $BDM^k$ space -- see
for example \cite{brezzi2012mixed}. The pressure space is given by
piecewise polynomials of one order less than the velocity space. By
this we have the property $\divergence(\Velspaceh) =
\Presspaceh$. Therefore, 
any weakly divergence-free 
velocity field is also strongly divergence free:
\begin{align} \label{eq::exactdivfree}
  (\divergence(\Velvarh), \Presvarhtest) = 0 \quad \Leftrightarrow \quad \divergence(\Velvarh) = 0 \quad\textrm{ in } \Omega.
\end{align}
Thus, 
any velocity field $u_h$ computed from the 
system~\eqref{eq::discrmixedstressstokesweak} is exactly
divergence free.

\section{Finite elements} \label{sec::finiteelements}

The aim of this section is to construct local finite elements that
yield the global finite element space $\Sigma_h$. We introduce
degrees of freedom (linear functionals) on each element which help us
impose the normal-tangential continuity. We also give an explicit
construction of a basis on a reference element and provide an
appropriate mapping to an arbitrary physical element of the
triangulation. This is especially useful for the implementation as there
is no need to compute a dual shape function basis by
biorthogonalization. The mapping technique permits easy extension to 
curved elements (although analysis of curved elements is beyond the
scope of this work). We then complete this   section by introducing an
interpolation operator that we shall use in the error analysis of the
next section.

The restriction of the function space $\Sigma_h$ defined
in~\eqref{eq:13} to a single element $T$ gives the local
finite element space 
$
\Sigma_k(T) := \left\{ \tau_h \in \Poly^k(T, \rr^{d \times d}):\;
  \trace{\tau_h} = 0, \; 
  (\tau_h)_{\normal\tangential} \in \Poly^{k-1}(F, \rr^{d-1}) 
  \text{ on all faces } F \in \facetT \right\},
$
where $ \facetT := \{F: F \subset \partial T \}$ is the set of element facets.
Let 
\[
  \dd := \{ M \in \rr^{d\times d}: (M : \id)  = 0\}.
\]
Then we may equivalently write
\begin{equation}
  \label{eq:14}
\Sigma_k(T) = \left\{ \tau_h \in \Poly^k(T, \dd):\;
  (\tau_h)_{\normal\tangential} \in \Poly^{k-1}(F, \rr^{d-1}) \text{ on
    all faces } F \in \facetT \right\}.  
\end{equation}
We proceed to study this space in detail, beginning with $\dd$.

\subsection{Trace-free matrices} \label{sec::tracefreematrices}

As a first step, we construct a basis for the space of matrices $\dd$
particularly suited to study normal-tangential components on facets.
Let
$V_i$, $i \in\mathcal{V}$, 
denote the vertices of $T$, 
where $\mathcal{V}:=\{0,1,2\}$ and $\mathcal{V}:=\{0,1,2,3\}$ in two
and three dimensions, respectively. Further let $F_i$ be the face
opposite to the vertex $V_i$ with the normal vector given by
$\normal_i$. 
The unit tangential vectors along edges are 
$\tangential_{ij} := (V_i - V_j) / |V_i - V_j|$. 
Finally let $\lambda_i$ be the unique barycentric coordinate function 
that equals one at the vertex $V_i$. When $d=2$, define 
three constant matrix functions, one for each $i \in \mathcal V$, 
\begin{align} \label{eq:basisface}
  S^i &:= \operatorname{dev}\big(\nabla \lambda_{i+1} \otimes \curl(
        \lambda_{i+2})\big)
\end{align}
where the indices $i+1$ and $i+2$ are taken modulo 3. When $d=3$, for
each $i \in \mathcal V$, we define the following two constant matrix functions
\begin{align} \label{eq:basisfacethree}
  S_0^i &:= \operatorname{dev}\big( \nabla \lambda_{i+1} \otimes (\nabla \lambda_{i+2}
          \times \nabla \lambda_{i+3})\big),   
          \quad S_1^i := \operatorname{dev}\big( \nabla
          \lambda_{i+2} \otimes (\nabla \lambda_{i+3} \times \nabla
          \lambda_{i+1})\big),
\end{align}
taking the indices $i+1$, $i+2$ and $i+3$ modulo~4.

\begin{lemma} \label{lem::normaltangentialprop} 
  The sets $\{ S^i : i \in \mathcal V\}$ and
  $\{ S^i_q : i \in \mathcal V, \; q =0,1 \}$ form a basis of $\dd$
  when $d=2$ and $3$, respectively.  Moreover, the normal-tangential
  component of $S^i$ and $S_q^i$ vanishes everywhere on the element
  boundary except on $F_i$, 
  \[
  S^i_{nt}|_{F_j}  = 0, 
  \qquad (S^i_q)_{nt}|_{F_j} = 0, \qquad i \ne j, \;\; F_j \in \facetT,
  \;\; i, j \in \mathcal V,
  \]
  while on $F_i$ it does not vanish. 
  When $i=j \in \mathcal V$
  and $d=3$,
  \begin{subequations}\label{eq:i=jcase}
  \begin{align}
    \tangential_{i+2, i+3}^\trans S^i_0 n_i     = 0, 
    \quad
    \tangential_{i+1,i+2}^\trans S^i_0 \normal_i \neq 0, 
    \quad 
    \tangential_{i+3,i+1}^\trans S^i_0 \normal_i \neq 0,
    \\
    \tangential_{i+2,i+3}^\trans S^i_1 \normal_i \neq 0,
    \quad
    \tangential_{i+1,i+2}^\trans S^i_1 \normal_i \neq 0,
    \quad 
    \tangential_{i+3,i+1}^\trans S^i_1 \normal_i  = 0.
  \end{align}
  \end{subequations}

\end{lemma}
\begin{proof}
  The first statement of the lemma follows once we prove the remaining
  statements.  Indeed, the linear independence of the given sets
  follows by examining their normal-tangential components
  facet-by-facet using the remaining statements. The spanning property
  follows by counting.

  To prove the remaining statements, 
  we start with the two dimensional case. We define
  \begin{align*}
    s_{i,j} = \Dev{\grad\lambda_i \otimes \curl(\lambda_j)}.
  \end{align*}
  Then $ s_{i+1, i+2} = S^i.$
Since the $nt$-component of the identity vanishes, 
  for any $p \in \mathcal V$ and
  any $t_p \in \curl(\lambda_p)$
  \[
  t_p^\trans  s_{i,j} n_p =
  t_p^\trans  
  \big[ \grad
  \lambda_i \otimes \curl( \lambda_j) \big] n_p
  = 
  (\nabla \lambda_i \cdot t_p) (\nabla
  \lambda_j \cdot t_p).
  \]
  All the stated properties in the two-dimensional case now follow
  easily from this identity together with the fact that $T$ is not
  degenerate.

  Next, consider the $d=3$ case. Let $s_{i,j,k} = \operatorname{dev}\big(
    \nabla \lambda_i \otimes (\nabla \lambda_j \times \nabla \lambda_k)
    \big)$. 
  If $i, j, k, l$ is any permutation of $\mathcal V,$
  by elementary manipulations, we see that for any $p \in \mathcal V$
  and any $t_p \in n_p^\perp$, 
  \begin{equation}
    \label{eq:15}
    t_p^\trans s_{i,j,k} n_p  = c (n_i \cdot t_p) (t_{il} \cdot n_p).
  \end{equation}
  for some $c\ne 0$. Therefore on any facet $F_p$, 
  we have $t_p^\trans(S^i_0)n_p = t_p^\trans(s_{i+1, i+2,
    i+3})n_p = c (n_{i+1} \cdot t_p) (t_{i+1, i} \cdot n_p)$ which
  vanishes for all $p \ne i$ since $n_{i+1}\cdot t_{i+1}=0$ and
  $t_{i+1, i} \cdot n_{i+2} =  t_{i+1, i} \cdot n_{i+3}
  =0$. Similarly, we conclude that $(S^i_1)_{nt} = 0 $ on all facets 
  except $F_i$. Since~\eqref{eq:15} also implies 
  \[
  t_{jk}^\trans s_{i,j,k} n_l =0, 
  \quad 
  t_{ki}^\trans s_{i,j,k} n_l \ne 0, 
  \quad 
  t_{ji}^\trans s_{i,j,k} n_l \ne 0,
  \]
  the statements in~\eqref{eq:i=jcase}  also follow.

  \end{proof}

\subsection{Normal-tangential bubbles}

Let the element  space of interior normal-tangential bubbles be
defined by 
\[
  \mathcal{B}_k(T) := \left\{ \tau_h \in \Stressspace_k(T): (\tau_h)_{\normal \tangential} =0 \right\}.
\]

\begin{lemma} \label{lem::structurebubblespace}
  Any $b \in \mathcal{B}_k(T)$ 
  can be expressed as either 
  \begin{align}
    \label{eq:16}
    b = \sum\limits_{i \in \mathcal{V}} \mu_i \lambda_i S^i
    \quad \textrm{or} \quad      
    b = \sum_{q=0}^1\;
    \sum\limits_{i \in \mathcal{V}} \mu^q_i\; \lambda_i S_q^i,
  \end{align}    
  for $d=2$ or $3$, respectively, where 
  $\mu_i,\mu^0_i,\mu^1_i \in \Poly^{k-1}(T)$. Consequently, 
  \[
  \dim \mathcal{B}_k(T) = 
  \left\{ 
    \begin{aligned}
      & \frac{3}{2} k (k+1),  && \text{ if } d=2, 
      \\
      &\frac{8}{6} k (k+1)(k+2),  && \text{ if } d=3.
    \end{aligned}
  \right.
  \]
\end{lemma}
\begin{proof}      
  We only show the proof in the $d=2$ case as the $d=3$ case is
  similar.  By Lemma~\ref{lem::normaltangentialprop} applied to the
  matrix $b(x)$, we obtain 
  \begin{equation}
    \label{eq:12}
      b(x) = \sum\limits_{i \in \mathcal{V}} a_i(x) S^i,
  \end{equation}
  and matching degrees, we conclude that $a_i \in \Poly^{k}(T)$.  Let
  $c_i $ equal the constant value of $S^i_{nt}|_{F_i}$, which is
  nonzero by Lemma~\ref{lem::normaltangentialprop}.
  Then $\tangential_i^\trans b(x) \normal_i = c_i a_i(x) = 0$ for all
  $x \in F_i$. Since $a_i(x)$ vanishes on $F_i$, it must take the form 
  $a_i(x) = \mu_i(x) \lambda_i(x) $ for
  some $\mu_i \in \Poly^{k-1}(T)$. 
  This proves~\eqref{eq:16}.
  
  The dimension count follows from~\eqref{eq:16}: again considering
  only the $d=2$ case, since $\mu_i \in \Poly^{k-1}(T)$ and
  $\{ \lambda_i S^i: i \in \mathcal V\}$ is a linearly independent
  set, the expansion in~\eqref{eq:16} shows that
  $\dim \mathcal{B}_k(T)$ is $3 \times \dim \Poly^{k-1}(T)$.
\end{proof}

\subsection{Mappings}

Suppose $\Tref$ is the unit simplex ($d=2$ or 3) and $T \in \mesh.$
Let $\phi_T: \Tref \rightarrow T$ be an  affine homeomorphism
and set $F_T := \phi_T'$. Due to the
shape regularity of the mesh,
\begin{align} \label{eq::jacobiscaling}
  || F_T ||_\infty \approx h \quad \textrm{ and } \quad || F_T^{-1} ||_\infty \approx h^{-1} \quad \textrm{ and } \quad |\mathrm{det}(F_T)| \approx h^d.
\end{align}
The proper transformation for functions in the $H(\divergence)$-conforming velocity space $\Velspaceh$ is the Piola transformation given by
$ \piola(\hat{\Velvar}_h) := (\mathrm{det} F_T)^{-1} F_T \hat{\Velvar}_h,
$
where $\hat{\Velvar}_h$ is a given polynomial on the reference element.
The Piola map preserves the normal components on facets, so is useful
for enforcing normal continuity.
For functions demanding tangential continuity, 
the proper transformation is the covariant transformation given by
$
  \covariant(\hat{\Velvar}_h) := F_T^{-\trans} \hat{\Velvar}_h. 
$
Therefore, to enforce the normal-tangential continuity required of
tensors in $\Stressspaceh,$ we combine the above  two transformations
and define
\begin{align}
  \label{eq:M}
  \covariantpiola(\hat{\Stressvar}_h) := \frac{1}{\mathrm{det}(F_T)}  F_T^{-\trans} \hat{\Stressvar}_h F_T^\trans, 
\end{align}
where $\hat{\Stressvar}_h \in \Sigma_k(\hat T).$
Of particular interest to us is how the normal-tangential components on
facets $F \in \facetT$ map. To study this, we use the restrictions of
the map $\phi_T$ to a reference facet $\hat F$ as well as to a
reference edge $\hat E$ (a $d-2$ subsimplex) in the $d=3$ case,
denoted by $\phi_T|_{\hat F}$ and $\phi_T|_{\hat{E}}$,
respectively. Their gradients are denoted by
$F_T^F = (\phi_T|_{\hat F})'$ and $F_T^E = (\phi_T|_{\hat{E}})'$.  In
the next result, $\hat n$ and $n$ denote the outward unit normals
vector on~$\hat F$ and $F$, respectively, while $\hat t$
denotes a unit tangent
vector along $\hat E$ (when $d=3$) or $\hat F$ (when $d=2$), and
similarly, $t$ denotes a
unit tangent 
vector along $E$ or $F$.

\begin{lemma} 
  \label{lem::nortangcont}
  Using the above notations and letting $\tau = \mathcal{M}(\hat
  \tau)$, we have 
  \[
  c \, \tangential^\trans \tau \normal
  =  \hat{\tangential}^{\trans}    \hat{\tau}  \hat{\normal},
  \qquad 
  \text{ where } 
  c = 
  \left\{
    \begin{aligned}
      & \det(F_T^F)^2  && \text{ if } d=2,
      \\
      & \det(F_T^F) \det(F_T^E)  && \text{ if } d=3.
    \end{aligned}
  \right.
  \]
  Furthermore, 
  \begin{align*}
    \trace{\hat{\tau}} = 0 \quad \Leftrightarrow \quad 
    \trace{\tau} = 0 .
  \end{align*}
\end{lemma}
\begin{proof}
  The unit normals and tangents on the reference and mapped
  configurations are related by

\begin{align*}
  \normal = \frac{\mathrm{det}(F_{T})}{\mathrm{det}(F_T^{F})} 
F_{T}^{-\trans} \hat{\normal} \quad \textrm{and} \quad
 \tangential = \frac{1}{\mathrm{det}(F_T^{E})} F_{T} \hat{\tangential},
\end{align*}
with the understanding that in two dimensions we should replace $F_T^E$ by $ F_T^F$. Then
\begin{align*}
  {\tangential}^\trans \tau {\normal} &= \frac{1}{\mathrm{det}(F_T^{E})} \hat{\tangential}^\trans F_{T}^\trans  \frac{1}{\mathrm{det}(F_{T})} F_{T}^{-\trans} \hat{\tau} F_{T}^\trans  \frac{\mathrm{det}(F_{T})}{\mathrm{det}(F_T^{F})} F_{T}^{-\trans} \hat{\normal} = \frac{1}{\mathrm{det}(F_T^{E})\mathrm{det}(F_T^{F})} \hat{\tangential}^\trans \hat{\tau}\hat{\normal}.
\end{align*}

Finally, the statement on traces follows from $\trace{F_T^{-\trans} \hat{\tau} F_T^\trans} = \trace{\hat{\tau}}.$
\end{proof}

\subsection{Definition of the finite element}

We define the local finite element in the formal style of
\cite{ciarlet2002finite} (also adopted in other texts, e.g., \cite{ernguermondfiniteelements, braess}) as a triple $(T, \Sigma_k(T), \Phi(T))$, where
the geometrical element $T$ is either a triangle or a tetrahedron, the
space $\Sigma_k(T)$  is defined by~\eqref{eq:14}, and $\Phi(T)$ is a set of linear
functionals representing the degrees of freedom defined as follows.
The first group of degrees of freedom is associated to the set of
element facets $\facetT,$
the set of $d-1$ subsimplices of $T$: 
for each $F \in \facetT$, define
\begin{align} \label{eq::facedofs}
    \Phi^F(\tau) :=
  \left\{ \int_F \tau_{\normal\tangential}\cdot r \ds: \; r \in
  \Poly^{k-1}(F, \rr^{d-1}) \right\}.
\end{align}
The next 
group is the set of interior  degrees of freedom  given by 
\begin{align} \label{eq::eldofs}
  \Phi^T  (\tau) := \left\{ \int_T \tau : F_T \hat \eta F_T^{-1}  \dx
  : \; \hat \eta \in
  \mathcal{B}_k(\hat T)  \right\}.
\end{align}
Then set 
\begin{equation}
  \label{eq:dofs}
  \Phi(T) := \Phi^T \cup \{ \Phi^F : \; F \in \facetT\}.
\end{equation}
We proceed to prove that
this set of degrees of freedom is unisolvent and that the  number of degrees
of freedom matches the dimension of 
$\Stressspace_k(T)$.

\begin{theorem} \label{thm:finiteelement}
  The triple $(T,\Stressspace_k(T), \Phi(T) )$ defines a finite
  element and 
  \[
  \dim (\Sigma_k(T) ) = 
  \left\{
    \begin{aligned}
      & \frac 3 2 (k+1)(k+2) - 3,  && \text { if } d=2, 
      \\
      & \frac 8 6 (k+1)(k+2)(k+3) - 8(k+1), && \text { if } d=3.
    \end{aligned}
  \right.
  \]
\end{theorem}
\begin{proof}
  To prove the unisolvency of the degrees of freedom, consider 
  a  $\tau_h \in \Stressspace_k(T)$ satisfying 
  $\phi(\tau_h) = 0$ for all $\phi \in \Phi(T)$. 
  As $(\tau_h)_{\normal\tangential} \in \Poly^{k-1}(F,
  \rr^{d-1})$ 
  the facet degrees of freedom 
   $\phi(\tau_h) = 0$ imply that $\tau_h \in
   \mathcal{B}_k(T)$. The interior degrees of freedom then yield
   \begin{align*}
     0  = \int_T \tau_h : F_T \hat \eta  F_T^{-1}
     = 
     \int_T F_T^\trans \tau_h F_T^{-\trans} : \hat \eta  
     = 
     \int_T (\det F_T)^{-1} \covariantpiola^{-1}( \tau_h) : \hat \eta  
     = 
     \int_{\hat T} 
     \covariantpiola^{-1}( \tau_h) : \hat \eta  
   \end{align*}
   for all $\hat \eta \in \mathcal{B}_k(\hat T)$. 
   By Lemma~\ref{lem::nortangcont}, $\covariantpiola^{-1}( \tau_h)$ is
   in $\mathcal{B}_k(\hat T)$, so this yields $\covariantpiola^{-1}(
   \tau_h) = 0$ and thus $\tau_h =0$.

   It only remains to prove the dimension count.  
   The dimension of 
   $\Stressspace_k(T)$ is given by
   $\dim \Poly^k(T,\dd)$ minus the
   number of linearly independent conditions represented by the
   constraints
   $ (\tau_h)_{\normal\tangential} \in \Poly^{k-1}(F, \rr^{d-1})
   $ for all $F \in \facetT$
   that every $\tau_h \in \Sigma_k(T)$ must satisfy. 
   Therefore,
   \begin{align*}
    \dim (\Stressspace_k(T))
     & \ge 
       \dim \Poly^k(T,\dd)  - 
       \dim \big[ \Poly^{k}(F, \rr^{d-1}) \setminus \Poly^{k-1}(F, \rr^{d-1})\big]
     \\ 
     &  = (d^2-1)  \dim \Poly^k(T) -  (d+1) (d-1) 
       \dim\big[  \Poly^{k}(F) \setminus \Poly^{k-1}(F)\big].
   \end{align*}
   Let $N_{\Sigma_k}$ denote the number on the right hand side. Using
   Lemma~\ref{lem::structurebubblespace} to count the number of
   degrees of freedom in $\Phi(T)$, we find that it coincides with
   $N_{\Sigma_k}.$ Since $N_{\Sigma_k}$ linear functionals on
   $\Sigma_k(T)$ are unisolvent, we conclude that
   $\dim (\Sigma_k(T)) = N_{\Sigma_k},$ which after simplification
   agrees with the statement of the theorem.
\end{proof}

\subsection{Construction of shape functions}

In view of the previous results, we can now write down shape
functions in barycentric coordinates. Its not difficult  to see that 
on any triangle $T$, the  set of
functions
\begin{equation}
  \label{eq:24}
\lambda_{i+1}^{\alpha_1}\lambda_{i+2}^{\alpha_2} S^i, 
\qquad 
\lambda_i^{\beta_0} \lambda_{i+1}^{\beta_1}\lambda_{i+2}^{\beta_2}
(\lambda_i S^i),   
\end{equation}
for all $i \in \mathcal{V},$ and all multi-indices $(\alpha_1, \alpha_2)$ and
$(\beta_0, \beta_1, \beta_2)$, with $\alpha_i \ge 0, \; \beta_i\ge 0$
having length $\alpha_1 + \alpha_2= \beta_0 + \beta_1 + \beta_2 =k-1$,
form a basis for $\Sigma_k(T)$. Similarly, when $T$ is a tetrahedron,
the following set is a basis for $\Sigma_k(T)$: 
\begin{equation}
  \label{eq:18} 
\lambda_{i+1}^{\alpha_1}\lambda_{i+2}^{\alpha_2} \lambda_{i+3}^{\alpha_3} 
S^i_q, 
\qquad 
\lambda_i^{\beta_0} \lambda_{i+1}^{\beta_1}\lambda_{i+2}^{\beta_2}\lambda_{i+3}^{\beta_3}
(\lambda_i S^i_q), 
\end{equation}
for all $i \in \mathcal{V},$ $q =0, 1$, and all multi-indices
$(\alpha_1, \alpha_2, \alpha_3)$ and
$(\beta_0, \beta_1, \beta_2, \beta_3)$, with
$\alpha_i \ge 0, \; \beta_i\ge 0$ having length
$\alpha_1 + \alpha_2 + \alpha_3= \beta_0 + \beta_1 + \beta_2 + \beta_3
=k-1$.
Instead of proving the linear independence of functions
in~\eqref{eq:24} or~\eqref{eq:18}, in the remainder of this section, 
we opt
to do so for another set of reference element shape functions
that we have implemented. By using a Dubiner basis instead of
barycentric monomials, the ensuing construction produces better
conditioned matrices.

We start by defining some basic notations needed for the
construction. 
The reference element is given by
\begin{alignat*}{3}
  \Tref &:= \{ (x_1,x_2) \in \rr^2: 0 \le x_1,x_2 \textrm{ and } x_1+x_2 \le 1 \} & \qquad\textrm{for} &\qquad d=2,\\
  \Tref &:= \{ (x_1,x_2,x_3) \in \rr^3: 0 \le x_1,x_2,x_3~\textrm{ and } x_1+x_2+x_3 \le 1 \} &\qquad\textrm{for} & \qquad  d=3.
\end{alignat*}
For $d=2$ we further define the reference faces and the corresponding normal and tangential vectors (see left picture in Figure \ref{fig::referenceelement}) by
\begin{align*}
  \begin{aligned}
  \hat{F}_0 &= \{ (x_1,x_2) \in \rr^2: 0 \le x_1,x_2 \le 1, x_1+x_2 = 1\}, \quad &&\normalhat_0:= \frac{1}{\sqrt{2}}(1,1)^\trans, \quad &&&\tangentialhat_0:= \frac{1}{\sqrt{2}}(-1,1)^\trans,  \\
  \hat{F}_1 &= \{ (0,x_2) \in \rr^2: 0\le x_2 \le 1\}, \quad &&\normalhat_1:= (-1,0)^\trans, \quad &&&\tangentialhat_1:=(0,-1)^\trans, \\
  \hat{F}_2 &= \{ (x_1,0) \in \rr^2: 0\le x_1 \le 1\}, \quad &&\normalhat_2:= (0,-1)^\trans, \quad &&&\tangentialhat_2:=(1,0)^\trans.
\end{aligned}
\end{align*}
For the three dimensional case we have
\begin{align*}
  \hat{F}_0 &= \{ (x_1,x_2,x_3) \in \rr^3: 0\le x_1,x_2,x_3 \le 1,  x_1+x_2+x_3 = 1\}, \\
  \hat{F}_1 &= \{ (0,x_2,x_3) \in \rr^3: 0\le x_2,x_3 \le 1, 0 \le x_2+x_3 \le 1\}, \\
  \hat{F}_2 &= \{ (x_1,0,x_3) \in \rr^2: 0\le x_1,x_3 \le 1, 0 \le x_1+x_3 \le 1\}, \\
  \hat{F}_3 &= \{ (x_1,x_2,0) \in \rr^2: 0\le x_1,x_2 \le 1, 0 \le x_1+x_2 \le 1\},
\end{align*}
with the associated normal and tangential vectors (see right picture in Figure \ref{fig::referenceelement}) 
\begin{align*}
\begin{aligned}
  \normalhat_0:= \frac{1}{\sqrt{3}}(1,1,1)^\trans, \quad &&\tangentialhat_{01}:= \frac{1}{\sqrt{2}}(-1,1,0)^\trans, \quad &&&\tangentialhat_{02}:= \frac{1}{\sqrt{2}}(0,1,-1)^\trans,  \\
  \normalhat_1:= (-1,0,0)^\trans, \quad &&\tangentialhat_{11}:= (0,-1,0)^\trans, \quad &&&\tangentialhat_{12}:= (0,0,-1)^\trans,  \\
  \normalhat_2:= (0,-1,0)^\trans, \quad &&\tangentialhat_{21}:= (1,0,0)^\trans, \quad &&&\tangentialhat_{22}:= (0,0,-1)^\trans,  \\
  \normalhat_3:= (0,0,-1)^\trans, \quad &&\tangentialhat_{31}:= (1,0,0)^\trans, \quad &&&\tangentialhat_{32}:= (0,-1,0)^\trans.
\end{aligned}
\end{align*}
\begin{figure}
  \begin{center}
    \begin{minipage}{0.45\textwidth}
       \begin{tikzpicture}
         \node[] (tref2d) at (0,0) {\includegraphics[width=0.85\textwidth,trim=1cm 0 0 0]{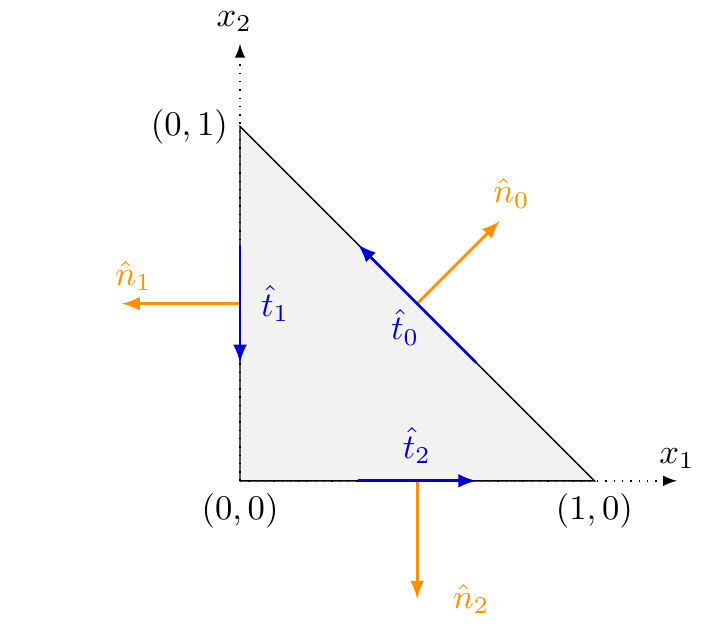}};
         \end{tikzpicture}
       \end{minipage}       
    \begin{minipage}{0.45\textwidth}
 \begin{tikzpicture}
   \node[] (tref2d) at (0,0) {\includegraphics[width=1\textwidth,trim=0.5cm 0 0 0]{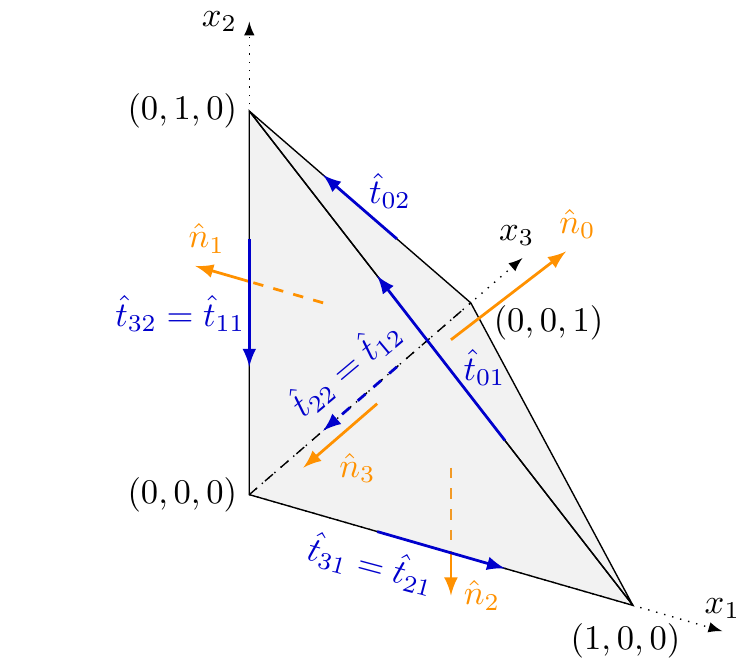}};
 \end{tikzpicture}
\end{minipage}
\end{center}
\vspace{-0.5cm}
\caption{The reference element and the corresponding normal and tangential vectors in two and three space dimensions.} \label{fig::referenceelement}
\end{figure}

In Section \ref{sec::tracefreematrices} we presented the construction of element wise constant matrices. Applying these techniques on the reference element (including a scaling with a proper constant) we derive for $d=2$ the matrices given by 
\begin{align}\label{eq:refbasisface}
  \hat{S}^0 := {\sqrt{2}} \begin{pmatrix} -1 & 0 \\ 0 & 1 \end{pmatrix} \quad \textrm{and} \quad \hat{S}^1 := \begin{pmatrix} 0.5 & 0 \\ 1 & -0.5 \end{pmatrix} \quad \textrm{and} \quad \hat{S}^2 := \begin{pmatrix} 0.5 & -1 \\ 0 & -0.5 \end{pmatrix},
\end{align}
and for $d=3$ the matrices
\begin{align} \label{eq:refbasisfacethree}
  \begin{aligned}
    \hat{S}^{0}_0 &= \sqrt{6}\begin{pmatrix}  \frac{-2}{3} & 0 & 0 \\ 0 & \frac{1}{3} &0 \\ 0 & 0 & \frac{1}{3} \end{pmatrix},    
    \hat{S}^{1}_0 =\begin{pmatrix}  \frac{1}{3} & 0 & 0 \\ 1 & \frac{-2}{3} &0 \\ 0 & 0 & \frac{1}{3} \end{pmatrix},
    \hat{S}^{2}_0 = \begin{pmatrix} \frac{-2}{3} & 1 & 0 \\ 0 & \frac{1}{3} &0 \\ 0 & 0 & \frac{1}{3} \end{pmatrix},
    \hat{S}^{3}_0 = \begin{pmatrix} \frac{-2}{3} & 0 & 1 \\ 0 & \frac{1}{3} &0 \\ 0 & 0 & \frac{1}{3} \end{pmatrix},
    \\
    \hat{S}^{0}_1 &= \sqrt{6}\begin{pmatrix} \frac{1}{3} & 0 & 0 \\ 0 & \frac{1}{3} &0 \\ 0 & 0 & \frac{-2}{3} \end{pmatrix},
    \hat{S}^{1}_1 = \begin{pmatrix} \frac{1}{3} & 0 & 0 \\ 0 & \frac{1}{3} &0 \\ 1 & 0 & \frac{-2}{3} \end{pmatrix},    
    \hat{S}^{2}_1 = \begin{pmatrix} \frac{1}{3} & 0 & 0 \\ 0 & \frac{1}{3} &0 \\ 0 & 1 & \frac{-2}{3} \end{pmatrix},    
    \hat{S}^{3}_1 = \begin{pmatrix} \frac{1}{3} & 0 & 0 \\ 0 & \frac{-2}{3} &1 \\ 0 & 0 & \frac{1}{3} \end{pmatrix}.
   \end{aligned}
\end{align}
Note that in order to follow the ideas described in Section \ref{sec::tracefreematrices} we took a particular choice of the numbering of the vertices of $\hat{T}$ and the corresponding tangential vectors.
Similar as in Lemma \ref{lem::normaltangentialprop}, a elementary calculations show that 
\begin{align} \label{ntorthogonalref}  
  \begin{aligned}   
  \tangentialhat^\trans_j\hat{S}^i \normalhat_j &= \delta_{ij}    &&\textrm{and}  &&&\tangentialhat^\trans_j \lambda_i\hat{S}^i \normalhat_j =0  &&&&\textrm{for}  &&&&&i,j =0,1,2,\\
  \tangentialhat^\trans_{jl}\hat{S}_q^{i} \normalhat_{j} &= \delta_{ij}\delta_{ql}  &&\textrm{and}  &&&\tangentialhat^\trans_{jl} \lambda_i\hat{S}_q^{i} \normalhat_{j} =0   &&&&\textrm{for} &&&&&i,j \in {0,1,2,3} \quad \textrm{and} \quad  q,l = 0,1.
\end{aligned}
\end{align}
and that $\{\hat{S}^i: i=0,1,2\}$ and $\{\hat{S}_q^i: i=0,1,2,3; q=0,1\}$ is a basis for $\dd$ in two and three dimensions, respectively. Based on these constant matrices we now construct shape function for the local stress space $\Stressspace_k(\hat{T})$.

We start with the two diemsnional case. Let $l_i(x_1)$ be the Legendre polynomial of order $i$ and let $l_i^S(x_1,x_2) := x_2^il_i(x_1/x_2)$ be the scaled Legendre polynomial of order $i$. Further let $p_i^j(x_1)$ be the Jacobi polynomial of order $i$ with coefficients $\alpha = j$, $\beta=0$. For a detailed definition we refer to the works \cite{Abramowitz,andrews}. We then define
\begin{align} \label{eq::dubinertwo}
  \hat{r}_{ij}(\lambda_\alpha, \lambda_\beta, \lambda_\gamma) := l_i^S(\lambda_\beta -\lambda_\alpha, \lambda_\alpha + \lambda_\beta)  p_j^{2i+1}(\lambda_\gamma -\lambda_\alpha -\lambda_\beta).
\end{align}
The polynomials $\hat{r}_{ij}(\lambda_\alpha, \lambda_\beta, \lambda_\gamma)$ with $0 \le i+j \le k$ and an arbitrary permutation $(\alpha,\beta,\gamma)$ of $(0,1,2)$ form a basis of the polynomial space $\Poly^k(\hat{T}, \rr)$. Next note that $ p_0^{2i+1}$ is constant, thus $\hat{r}_{ij}(\lambda_\alpha, \lambda_\beta, \lambda_\gamma)=\hat{r}_{i0}(\lambda_\alpha, \lambda_\beta)$. Then there holds that for $0 \le i \le k$ the restriction of the polynomials $\hat{r}_{i0}(\lambda_{j+1}, \lambda_{j+2})|_{\hat{F}_j}$, where the indices $j+1$ and $j+2$ of the barycentric coordinate functions are taken modulo $3$, form a basis of the polynomial space $\Poly^k(\hat{F}_j,\rr)$ (see chapter 3.2 in \cite{karniadakis2013spectral} or in \cite{dubiner1991spectral}). By this we define a local basis of the stress space by
\begin{align*}
  \hat{\Psi}_k^F&:= \{ \hat{S}^j \hat{r}_{i0}(\lambda_{j+1}, \lambda_{j+2}): j=0,1,2  \textrm{ and } 0 \le i \le k-1 \}, \\ 
  \hat{\Psi}^T_k &:= \{ \lambda_j \hat{S}^j \hat{r}_{il}(\lambda_0,\lambda_1,\lambda_2): j=0,1,2 \textrm{ and } 0 \le i+l \le k-1\}. 
\end{align*}
For $d = 3$ we define similar as before
\begin{align} \label{eq::dubinerthree}
  \hat{r}_{ijl}(\lambda_\alpha,&\lambda_\beta, \lambda_\gamma,  \lambda_\delta) \\
                          &:= l_i^S(\lambda_\beta -\lambda_\alpha, \lambda_\alpha + \lambda_\beta)  p_j^{2i+1,S}(\lambda_\gamma -\lambda_\alpha -\lambda_\beta, \lambda_\gamma +\lambda_\alpha +\lambda_\beta ) p_l^{2i+2j+2,S}(\lambda_\delta -\lambda_\alpha -\lambda_\beta-\lambda_\gamma),\nonumber
\end{align}
where $p_i^{j,S}(x_1,x_2):= x_2^i p_i^{j}(x_1/x_2)$ is the scaled Jacobi polynomial. Again we have that $\hat{r}_{ijl}(\lambda_\alpha,\lambda_\beta, \lambda_\gamma,  \lambda_\delta)$ with $0 \le i+j+l \le k$ and an arbitrary permutation $(\alpha,\beta,\gamma,\delta)$ of $(0,1,2,3)$ defines a basis for $\Poly^k(\hat{T},\rr)$ and that for $0 \le  i + l \le k$ the restriction $\hat{r}_{il0}(\lambda_{j+1}, \lambda_{j+2}, \lambda_{j+3})|_{\hat{F}_{j}}$ is a basis of $\Poly^k(\hat{F}_{j},\rr)$ where the indices of the barycentric coordinate functions are now taken modulo 4. By this we define the local basis on the reference tetrahedron by
\begin{align*}
  \hat{\Psi}^F_k &:= \{ \hat{S}_q^{j} \hat{r}_{il0}(\lambda_{j+1}, \lambda_{j+2}, \lambda_{j+3}):j=0,1,2,3 \textrm{ and } q=0,1 \textrm{ and }   0 \le i+l \le k-1 \} \\ 
  \hat{\Psi}^T_k &:= \{ \lambda_j  \hat{S}_q^{j}  \hat{r}_{ilg}(\lambda_0,\lambda_1,\lambda_2,\lambda_3): j=0,1,2,3 \textrm{ and } q=0,1  \textrm{ and }0 \le i+l+g \le k-1 \}. 
\end{align*}
\begin{theorem}
  The set of functions $\{\hat{\Psi}_k^F \cup  \hat{\Psi}_k^T\}$ is a basis for $\Stressspace_k(\hat{T})$.
\end{theorem}
\begin{proof}
  We start with the two dimensional case. An elementary calculation shows that the functions $\lambda_i\hat{S}^i$ with $i=0,1,2$ are linearly independent.
  Let $\alpha_i^j \in \rr$ and $\beta_{il}^j \in \rr$ be arbitrary coefficients and define $\hat{S}^j_i:= \hat{S}^j \hat{r}_{i0}(\lambda_{j+1}, \lambda_{j+2})$ and $\hat{B}^j_{il}:= \lambda_j \hat{S}^j \hat{r}_{il}(\lambda_0,\lambda_1,\lambda_2)$. We assume that
  \begin{align*}
    \sum\limits_{j=0}^2\sum\limits_{i=0}^{k-1} \alpha_i^j \hat{S}^j_i + \sum\limits_{j=0}^2\sum\limits_{i=0}^{k-1}\sum\limits_{l=i}^{k-1} \beta_{il}^j \hat{B}^j_{il} = \begin{pmatrix} 0 & 0 \\ 0 & 0  \end{pmatrix},
  \end{align*}
  and show that this induces that all coefficients are equal to zero. This then proves the linear independency of $\{\hat{\Psi}_k^F \cup  \hat{\Psi}_k^T\}$. Let $\hat{F}_g$ with $g = 0,1,2$ be an arbitrary reference face. Due to \eqref{ntorthogonalref}, there holds
  \begin{align*}
    \tangentialhat^T_g\left(\sum\limits_{j=0}^2\sum\limits_{i=0}^{k-1} \alpha_i^j \hat{S}^j_i + \sum\limits_{j=0}^2\sum\limits_{i=0}^{k-1}\sum\limits_{l=i}^{k-1} \beta_{il}^j \hat{B}^j_{il}\right)\normalhat_g
    = \tangentialhat^T_g\left(\sum\limits_{i=0}^{k-1} \alpha_i^g \hat{S}^g_i\right)\normalhat_g
    = \tangentialhat^T_g\left(\sum\limits_{i=0}^{k-1} \alpha_i^g \hat{S}^g \hat{r}_{i0}(\lambda_{g+1}, \lambda_{g+2}) \right)\normalhat_g = 0. 
  \end{align*}
  As $\hat{r}_{i0}(\lambda_{g+1}, \lambda_{g+2})$ is a polynomial basis on $\hat{F}_g$, and $\hat{S}^g, \normalhat_g$ and $\tangentialhat_g$ are constant it follows that all coefficients $\alpha_i^g$ have to be zero. As $g$ was arbitrary we conclude $\alpha_i^j=0$ for $j=0,1,2$ and $0 \le i \le k-1$.

  As the functions $\lambda_i\hat{S}^i$ are linearly independent we have for each $g = 0,1,2$ (due to the assumption at the beginning)
  \begin{align*}
\sum\limits_{i=0}^{k-1}\sum\limits_{l=i}^{k-1} \beta_{il}^g \hat{B}^g_{il}  = \sum\limits_{i=0}^{k-1}\sum\limits_{l=i}^{k-1} \beta_{il}^g \hat{r}_{il} \lambda_g \hat{S}^g = \begin{pmatrix} 0 & 0 \\ 0 & 0  \end{pmatrix}.
  \end{align*}
  As $\hat{r}_{il} \lambda_g$ is a basis for $ \lambda_g \Poly^{k-1}(\hat{T})$, and the last equation holds true for all points in $\hat{T}$ we conclude $\beta_{il}^g = 0$ for $0 \le i+l \le k-1$. As $g$ was arbitrary we conclude that all coefficients are equal to zero. Note that by $\trace{S^i} = 0$, all shape function in  $\{\hat{\Psi}_k^F \cup  \hat{\Psi}_k^T\}$ are trace free and are further tensor valued polynomials up to order $k$. Further the normal tangential trace is only a polynomial up to order $k-1$ thus all shape functions belong to $\Stressspace_k(\hat{T})$. Counting the dimensions we have by Theorem \ref{thm:finiteelement}
  \begin{align*}
    \big| \hat{\Psi}_k^F\big| + \big| \hat{\Psi}_k^T \big| = 3k + \frac{3 k(k+1)}{2} =N_{\Stressspace_k},    
  \end{align*}
  what concludes the proof. In three dimensions we proceed similar. The linearly independence can be shown with the same steps. Further with the same arguments all shape functions belong to $\Stressspace_k(\hat{T})$. Again by Theorem \ref{thm:finiteelement} and
  \begin{align*}
    \big| \hat{\Psi}_k^F\big| + \big| \hat{\Psi}_k^T \big|
    = 8 \frac{k(k+1)}{2} + 8 \frac{k(k+1)(k+2)}{6}  =N_{\Stressspace_k}, 
  \end{align*}
  we conclude the proof.
\end{proof}

\begin{remark} \label{rem::basis}
  Note how the basis was separated into shape functions associated to
  faces ($\hat{\Psi}^F_k$) and shape functions associated to the element
  interior ($\hat{\Psi}^T_k$). The polynomial degrees in each group can be
  separately chosen to construct a variable-degree global finite
  element space (e.g., for $hp$ adaptivity). E.g., the span of the union of 
  $\Psi^F_{k_1}$ and 
  $\Psi^T_{k_2}$ gives an element space that has normal-tangential
  trace of degree $k_1-1$ and inner (bubble) shape functions of degree $k_2$.
\end{remark}

\subsection{Construction of a global basis}


Using the local basis on the reference triangle $\hat{T}$ we can now simply define a global basis for the stress space $\Stressspaceh$. This is done in the usual way. Using the  mapping $\covariantpiola$ and a basis function $\hat{S} \in \{ \hat{\Psi}_k^T \cup \hat{\Psi}_k^F \}$ we define the restriction of a global shape function $S$ (with support on a patch) on an arbitrary physical element $T \in \mesh$ by
\begin{align*}
  S := \covariantpiola(\hat{S}).
\end{align*}

Next we identify all topological entities, vertices and faces, of the physical element $T$ with the corresponding entities of the global mesh. This identification is needed as faces and vertices coincide for adjacent physical elements. Note that the global orientation of the faces (and edges) plays an important role in order to assure (normal-tangential) continuity. This is a well known difficulty: see \cite{zaglmayr2006high} for a detailed discussion regarding this topic. By this we construct global basis functions which are, restricted on a physical element $T \in \mesh$, always a  mapped basis function of the basis defined on the reference element $\hat{T}$.

Further note that due to Lemma \ref{lem::nortangcont} the resulting basis functions are normal tangential continuous, thus $[\![ S_{\normal \tangential} ]\!]=0$.  To see this let $\phi_1$ be the mapping of an arbitrary element $T_1$ and let $\phi_2$ be the mapping of an element $T_2$ such that $F = T_1 \cap T_2$. There exists a reference face $\hat{F} \subset \partial \hat{T}$ such that $F = \phi_1(\hat{F}) = \phi_1(\hat{F})$ (in the sense of a set) and $\phi_1|_{\hat{F}} = \phi_2|_{\hat{F}}$ (in the sense of equivalent functions). By this, and the same ideas for an reference edge $\hat{E}$ in the three dimensional case, the constant $c$ in Lemma \ref{lem::nortangcont} is the same for both mappings. In two dimensions we have the identity $S_{\normal\tangential} =  ({\tangential}^\trans S {\normal}) t$, thus Lemma \ref{lem::nortangcont} implies normal-tangential continuity of $S$ because $S$ was a mapped basis functions of the reference element. In three dimensions $S_{\normal\tangential}$ is a tangent vector in $F$. Each tangent vector can be represented as a linear combination of two arbitrary edge tangent vectors $t_i \subset \partial F$. By Lemma \ref{lem::nortangcont} we deduce that the scalar values ${\tangential_i}^\trans S {\normal}$ are preserved, thus again we have normal tangential continuity. Taking all functions in $\{ \hat{\Psi}_k^T \cup \hat{\Psi}_k^F \}$ and mapping them to each element separately results in a basis for
$\Stressspaceh$.

\subsection{An interpolation operator for the stress space.}

We finish this section by introducing an interpolation operator for the stress space and showing an approximation result.
Using the global  degrees of freedom of ${\Stressspaceh}$ 
a canonical interpolation operator 
 $I_{\Stressspaceh}$ can be defined as usual. On each $T \in \mesh$,
 the interpolant $(I_{\Stressspaceh}\sigma)|_T$
coincides  with the canonical local interpolant $I_T(\sigma|_T)$ defined, as
 usual, using the local degrees of freedom in $\Phi(T)$, by 
 \begin{equation}
   \label{eq:20}
    \phi(\sigma - I_T \sigma)= 0 \quad \text{ for all } \phi \in \Phi(T).
 \end{equation}
 Recalling  the map $\MM$ from~\eqref{eq:M}, note that 
 $\covariantpiola^{-1}(\sigma) = \mathrm{det}(F_T^{F}) F_F^\trans \sigma
 F_T^{-\trans}$. 

 \begin{lemma} \label{lem:Imap} 
   For any $\sigma \in H^1(T, \RRR^{d \times d}),$
  \[
  \MM^{-1} (I_T \sigma) = I_{\hat T} (\MM^{-1} (\sigma)).
  \]
\end{lemma}
\begin{proof}
  Since both the left and right hand sides are in $\Sigma_k(\hat T)$, 
  it suffices to prove that 
  \begin{align}
    \label{eq:21}
    \hat{\phi}(  \covariantpiola^{-1} (I_T\sigma)-
    I_{\hat{T}}(\covariantpiola^{-1} \sigma))= 0 \quad \text{ for all
    }
    \hat{\phi} \in \Phi(\hat{T}).
  \end{align}  
  
  To see that~\eqref{eq:21} holds for the interior degrees of freedom
  on $\hat T$ as defined in~\eqref{eq::eldofs}, noting that
  $F_{\hat T}$ is the identity, we have for all
  $\hat \eta \in \mathcal{B}_k(\hat T),$
  \begin{align*}
    \int_{\hat{T}}
    \big[ \MM^{-1} (I_T \sigma) - I_{\hat T}(\MM^{-1} \sigma) \big]
    : F_{\hat T} \hat \eta F_{\hat T}^{-1}  \dxhat 
    & =  
    \int_{\hat{T}}
    \big[ \MM^{-1} (I_T \sigma) - \MM^{-1} \sigma \big]
    : \hat \eta  \dxhat 
    \\
    & = \int_{{T}}
    (I_T \sigma -  \sigma )
      : F_T \hat{\eta} F_T^{-1}  \dx  = 0 
  \end{align*}
  due to the equality of interior degrees of freedom on~$T$
  in~\eqref{eq:20}. 

  Next, consider the facet degrees of freedom. We only consider the
  $d=3$ case (as the other case is simpler). On an arbitrary facet
  $\hat F \in \mathcal{F}_{\hat T}$, 
  choose two arbitrary edges $\hat{E_1}, \hat{E_2}$
  with unit tangential vectors $\hat{\tangential}_1$ and
  $\hat{\tangential}_2$. 
  Using a dual tangential basis $\hat{s}_1$ and $\hat{s}_2$
  such $\hat{s}_i \cdot \hat{t}_i = \delta_{ij}$, we expand
  \begin{align*}
    [\covariantpiola^{-1} (I_T\sigma- \sigma) ]_{\normal\tangential}
    = [\hat{\tangential}_1^\trans \covariantpiola^{-1} (I_T\sigma- \sigma) \hat{\normal} ] \hat{s}_1 + [\hat{\tangential}_2^\trans \covariantpiola^{-1} (I_T\sigma- \sigma) \hat{\normal} ] \hat{s}_2.
  \end{align*}
  Next we choose arbitrary $\hat{r}_1, \hat{r}_2 \in
  \Poly^{k-1}(\hat{F}, \rr)$ and define 
  \[
  \hat{r}:= \frac{\hat{r}_1}{ \mathrm{det}(F_{E_1})} \hat{\tangential}_1 +
  \frac{\hat{r}_2}{ \mathrm{det}(F_{E_2})} \hat{\tangential}_2.
  \]
  Let $r_i = \hat r_i \circ \phi_T$. Using 
  a biorthogonal basis $s_1,s_2$ with respect to 
  unit tangents $\tangential_1$ and $\tangential_2$ of mapped edges
  $E_1$ and $E_2$, we  have 
  $r :=r_1 \tangential_1 + r_2 \tangential_2.$ 
  Using Lemma \ref{lem::nortangcont} we deduce  
  \begin{align*}
   [\covariantpiola^{-1} (I_T\sigma- \sigma) ]_{{\normal}{\tangential}} =  \mathrm{det}(F_T^{F})\mathrm{det}(F_{E_1})[{\tangential}_1^\trans (I_T\sigma- \sigma) {\normal} ] \hat{s}_1 + \mathrm{det}(F_T^{F})\mathrm{det}(F_{E_2})[\tangential_2^\trans (I_T\sigma- \sigma) {\normal} ] \hat{s}_2,
  \end{align*}
  so 
  \begin{align*}
    \int_{\hat{F}} [\covariantpiola^{-1} (I_T\sigma- \sigma) ]_{{\normal}{\tangential}} \cdot \hat{r} \dxhat &= \int_F [{\tangential}_1^\trans (I_T\sigma- \sigma) {\normal} ] r_1 s_1 \cdot \tangential_1\dx + \int_F [{\tangential}_2^\trans (I_T\sigma- \sigma) {\normal} ] r_2 s_2 \cdot \tangential_2 \dx \\
                                                                                                                     &= \int_F [ ({\tangential}_1^\trans (I_T\sigma- \sigma) {\normal}) s_1  + ({\tangential}_2^\trans (I_T\sigma- \sigma) {\normal}) s_2] \cdot[ r_1 \tangential_1 + r_2 \tangential_2] \dx \\
    &= \int_F (I_T\sigma- \sigma)_{\normal\tangential} \cdot r \dx =0
  \end{align*}
 where the last equality is due to the equality of the 
 facet degrees of freedom in~\eqref{eq:20}.
\end{proof}

\begin{theorem}[Interpolation operator for $\Stressspaceh$]
 \label{thm::interpolationop}
 For any $m \ge 1$ and any 
 $\sigma \in \{ \tau \in H^m(\mesh, \RRR^{d \times d}): \;
  \jump{\tau_{nt}} =0\}$, the interpolant
  $I_{\Stressspaceh}\sigma$ is well defined and
  there is a mesh-independent constant $C$ such that 
  \begin{align}
    || \sigma - I_{\Stressspaceh}\sigma ||_{L^2(\om)} + \sqrt{\sum\limits_{F \in \facets} h || (\sigma -
    I_{\Stressspaceh}\sigma)_{\normal\tangential} ||^2_{F}}
    \le \; C h^s ||\sigma||_{H^s(\mesh)}
    \label{eq::approxsigma}
  \end{align}
    for all $s \le  \min(k,m)$.
\end{theorem}
\begin{proof}
  Let $\hat \sigma = \MM^{-1}( \sigma|_{T})$.
  By Lemma~\ref{lem:Imap},
  $\MM^{-1} (\sigma - I_T \sigma) = \hat \sigma - I_{\hat T}
  \hat{\sigma}$.
  By the unisolvency of the reference element degrees of freedom
  (Theorem~\ref{thm:finiteelement}),
  \[
  \hat\sigma - I_{\hat T} \hat\sigma = 0 
  \quad \text{ for all } \hat\sigma \in \Poly^{k-1}(\hat T,
  \RRR^{d \times d}).  
  \]
  Now a standard argument using the Bramble-Hilbert lemma, the
  continuity of
  $I_{\hat T}: H^s(\hat T, \RRR^{d \times d}) \rightarrow L^2(\hat T,
  \RRR^{d \times d})$, and scaling arguments, finish the proof.
\end{proof}

\section{A priori error analysis} \label{sec::apriorianalysis}

In this section we show discrete inf-sup stability of the MCS method,
optimal error estimates (Theorem~\ref{th::optimalconvergence}) and
pressure robustness (Theorem~\ref{th::pressurerobust}).  The error
analysis is in the following norms.
\begin{align*}
  \Sigmanormh{ \tau_h }^2 &:= || \tau_h ||_{L^2(\om)}^2 = || \Dev{\tau_h} ||_{L^2(\om)}^2, 
  && \tau_h \in \Sigma_h,
  \\
  \Velnormh{ v_h }^2  &:= \honenormh{v_h}^2 := 
                     \sum\limits_{T \in \mesh} ||  \nabla v_h ||_{T}^2 
                      + \sum\limits_{F \in \facets} 
                      \frac{1}{h} \| \jump{{ (v_h)_\tangential}} \|^2_{F},
  && v_h \in V_h, 
  \\
  \Presnormh{ q_h}^2 &:= || q_h ||_{L^2(\om)}^2, 
                     && q_h \in Q_h.
\end{align*}
Comparing with (appropriate) norms of the infinite dimensional spaces $\Velspace$
and $\Stressspace$, these  norms might seem unnatural.
But we choose these norms
in order to obtain velocity error estimates in an $H^1$-like norm 
comparable to the standard velocity-pressure formulation. Since our
discrete spaces do not admit $H^1$-conformity,
our $\Velnormh{ \cdot }$-norm contains a term that penalizes 
the tangential discontinuities (as in the analysis of discontinuous Galerkin methods).
The $L^2$-like norm on the $\Sigma_h$ is also related to an $H^1$-like norm of the
velocity since we expect $\sigma_h$ to be an approximation 
of $\nu \grad u$.

\subsection{Norm equivalences}

We use $A \sim B $ to indicate that there are 
constants $c, C>0$ {\em independent of the mesh size $h$ and the viscosity $\visc$} 
such that $c A \le B \le C A$.  We also use
$A \lesssim B$ when there is a $C>0$ 
independent of $h$ and $\visc$ such that
$A \le C B$ (and $\gtrsim$ is defined similarly).  Due to quasiuniformity, 
the following estimates follow by standard scaling
arguments: for any $\hat \tau \in \Sigma_k(\hat T)$, letting $\tau =
\mathcal M ( \hat \tau)$, 
\begin{gather}
  \label{eq:scale-s_L2norm}
  h^{d}\| \tau_h \|_{T}^2 \sim \| \hat \tau_h \|_{0, \hat T}^2.
\end{gather}
On any $F \in \facetT$, Lemma~\ref{lem::nortangcont}, together with
a scaling argument yields
\begin{gather}
  \label{eq:scale-s_nt}
  h^{d+1} \left\| t^\trans\tau_h n \right\|_{F}^2 \sim
  \left\| \hat{t}^\trans \hat\tau_h {\hat n} \right\|_{0, \hat F}^2.
\end{gather}

\begin{lemma} \label{lem::normeqsigma}
  For all $\tau_h \in \Stressspaceh$, 
  \begin{align*}
    \Sigmanormh{ \tau_h }^2 \sim \sum\limits_{T \in \mesh} || \Dev{\tau_h}
    ||_{T}^2 + \sum\limits_{F \in \facets} h 
    \big\| \jump{(\tau_h)_{\normal\tangential}} \big\|_{F}^2.
  \end{align*}
\end{lemma}
\begin{proof}
  By finite dimensionality, for any face $\hat F \in \mathcal{F}_{\hat T}$,
  \[
    h \| \hat{t}^\trans\hat\tau_h {\hat n} \|_{0, \hat F}^2 \lesssim \| \hat
    \tau_h \|_{0, \hat T}^2, 
    \quad \text{ for all } \hat \tau_h \in \Sigma_k(\hat T).
  \]
  Due to   \eqref{eq:scale-s_nt} and~\eqref{eq:scale-s_L2norm}, this yields
  \[
  \sum_{F \in \facets} 
  h \big\| \jump{(\tau_h)_{n t}} \big\|_{F}^2 \lesssim 
  \sum_{T \in \mesh} \| \tau_h \|_{T}^2, 
  \quad \text{ for all } \tau_h \in \Sigma_k( T).
  \]
  This proves one side of the stated
  equivalence. The other side is obvious.
\end{proof}

On each facet $F \in \facets$ with normal vector $n_F$, let
$\proj_F^0$ denote the $L^2$ projection onto the space of constant
tangential vectors in $n_F^\perp$, i.e., for any vector function
$v \in L^2(F, n_F^\perp)$, the projection $\proj_F^0 v \in n_F^\perp$
satisfies $(\proj_F^0 v, t)_F = (v, t)_F$ for all $t \in n_F^\perp.$

\begin{lemma} \label{lem::normeqvel} 
  For all $v_h \in \Velspaceh$,
  \begin{align*}
     \Velnormh{ v_h }^2 \sim \sum\limits_{T \in \mesh}
    ||  \nabla v_h    ||_{T}^2 + \sum\limits_{F \in \facets} \frac{1}{h} 
    \big\| \proj_F^0\jump{{ (v_h)_\tangential}} \big\|^2_{F}
  \end{align*}
\end{lemma}
\begin{proof}
  One side of the equivalence is obvious from the continuity of
  $\proj_F^0$.  For the other direction,
  \begin{align} 
    \Velnormh{ v_h }^2
    & \le \sum\limits_{T \in \mesh}
      ||  \nabla v_h ||_{T}^2
      + \sum\limits_{F \in \facets} 
      \frac{2}{h}
      \big\| \proj_F^0\jump{{ (v_h)_\tangential}} \big\|^2_{F}
      + 
      \frac{2}{h}
      \big\| \jump{{ (v_h)_\tangential}} - \proj_F^0\jump{{ (v_h)_\tangential}} \big\|^2_{F}.
      \label{eq::insertprojection}
  \end{align}
  Now, on each facet $F \in \facetT,$ we use the standard estimate
  $
  \big\| (v_h)_\tangential - \proj_F^0 (v_h)_\tangential
  \big\|_{F}
  \lesssim h^{1/2} \| \nabla v_h \|_{T}
  $ to complete the proof.
\end{proof}

\subsection{Stability analysis}

\begin{lemma}[Continuity of $\ablf$, $\blfone$ and $\blftwo$]
  \label{th:continuity}
  The bilinear forms $\ablf$, $\blfone$ and $\blftwo$ are continuous:
  \begin{align*}
    \begin{aligned}
    \ablf (\Stressvarh ,\Stressvarhtest) &\cle \frac{1}{\sqrt{\visc}}  \Sigmanormh{ \Stressvarh }  \frac{1}{\sqrt{\visc}}\Sigmanormh{ \Stressvarhtest } && \quad\text{ for all } \Stressvarh ,\Stressvarhtest \in \Stressspaceh \\
    \blfone(\Velvarhtest, \Presvarh) &\cle \Velnormh{ \Velvarhtest } || \Presvarh ||_{Q_h} && \quad\text{ for all } \Velvarhtest \in \Velspaceh, \Presvarh \in \Presspaceh\\ 
    \blftwo(\Stressvarh, \Velvarhtest) &\cle \Sigmanormh{ \Stressvarh } \Velnormh{ \Velvarhtest } && \quad\text{ for all } \Stressvarh \in \Stressspaceh, \Velvarhtest \in \Velspaceh.
  \end{aligned}
  \end{align*}
\end{lemma}
\begin{proof}
  The continuity for the bilinear forms $\ablf$ and $\blfone$ follows
  from the Cauchy-Schwarz inequality, we only consider $b_2$, which 
  by~\eqref{eq::blftwoequitwo} can be written as
  \begin{align*}
\blftwo(\Stressvarh, \Velvarhtest) =  -\sum\limits_{T \in \mesh} \int_T \Stressvarh : \nabla \Velvarhtest \dx + \sum\limits_{F \in \facets} \int_F (\Stressvarh)_{nt} \cdot \jump{(\Velvarhtest)_t} \ds .
  \end{align*}
  Since $(\Stressvarh)_{\normal\tangential} =
  (\Dev{\Stressvarh})_{\normal\tangential}$, 
  we conclude the proof by  Cauchy-Schwarz inequality 
  and Lemma~\ref{lem::normeqsigma}.
\end{proof}

\begin{lemma}[Coercivity of $\ablf$ on the kernel] \label{th::coercivity}
Let
$\kernel := \{ (\Stressvarhtest,\Presvarhtest) \in  \Stressspaceh
\times \Presspaceh: \blfone(\Velvarhtest, \Presvarhtest) +
\blftwo(\Stressvarh,\Velvarhtest)= 0$ for all $ \Velvarhtest \in \Velspaceh \}.
$ 
For all $(\Stressvarh,\Presvarh) \in \kernel$, 
  \begin{align*}
    \frac{1}{\visc} 
    \big(\,\Sigmanormh{ \Stressvarh } + \Presnormh{ \Presvarh }
    \big)^2  \lesssim \;
    \ablf(\Stressvarh,\Stressvarh).
  \end{align*}
\end{lemma}
\begin{proof}
  Let $(\Stressvarh,\Presvarh) \in \kernel$ be arbitrary. 
  As $\visc^{-1}\Sigmanormh{ \Stressvarh }^2 =
  \ablf(\Stressvarh,\Stressvarh)$ it is sufficient to bound only the
  norm of $\Presvarh$. 
  It is well known -- see e.g., \cite{brezzi2012mixed} -- that for any 
  $\Presvarh \in \Presspaceh$
  \begin{equation}
    \label{eq:17}
    \exists   
    \Velvarhtest \in \Velspaceh: 
    \qquad \divergence(\Velvarhtest) =p_h, \quad
    \Velnormh{ \Velvarhtest } \lesssim ||    \Presvarh ||_{Q_h}.
  \end{equation}
  With this $v_h$, 
  \begin{align*}
    2 \Presnormh{ \Presvarh }^2 
    & = \sum\limits_{T \in \mesh} \int_T \Presvarh \Presvarh \dx =
      \sum\limits_{T \in \mesh} \int_T \divergence(\Velvarhtest)
      \Presvarh \dx 
      = \blfone(\Velvarhtest,\Presvarh)
      \\
    & = - \blftwo(\Stressvarh,\Velvarhtest) 
    && \text{as $(\Stressvarh,\Presvarh) \in \kernel$,}
    \\
    & 
      =  
      \sum\limits_{T \in \mesh} \int_T \Stressvarh : \nabla
      \Velvarhtest \dx 
      - \sum\limits_{F \in \facets} 
      \int_F (\Stressvarh)_{nt} \cdot \jump{(\Velvarhtest)_t} \ds
      && \text{by \eqref{eq::blftwoequitwo},}
      \\
    & \le 
      || \Dev{\Stressvarh} ||_{L^2(\om)} \Velnormh{ \Velvarhtest } 
    && \text{using Lemma~\ref{lem::normeqsigma},}
    \\
    & \cle \Sigmanormh{ \Stressvarh } || \Presvarh ||_{Q_h}
    && \text{by~\eqref{eq:17}}.
  \end{align*}
\end{proof}

Next, we proceed to verify the discrete LBB condition (in
Theorem~\ref{th::lbbcondition} below).  
Define 
\begin{gather*}
\Velspacehdivfree := \{ w_h \in \Velspaceh: \divergence(w_h) = 0\}, 
\\  
\| v_h \|_{1, \text{dev}, h} :=
  \left(
    \sum_{K \in \mesh} \| \dev{\nabla v_h}  \|_{T}^2  + 
    \sum_{F \in \facets} 
    \frac 1 h \left\| \jump{ (v_h)_t } \right\|_{F}^2
   \right)^{1/2}.
\end{gather*}
Since
$\|\nabla v_h \|_{T}^2 \sim \| \dev{\nabla v_h} \|_{T}^2 + \| \div(v_h)
\|_{T}^2$
on any $T \in \mesh$ and for any $v_h \in V_h$, 
we have 
\begin{equation}
  \label{eq:equivVh}
  \devnorm{v_h} \sim \| v_h \|_{V_h} \quad \text{ for all } v_h
\in V_h^0.
\end{equation}
A first step towards proving the LBB condition is the construction of
a  specific stress function $\tau_h$ which only depends on
$\Dev{\nabla v_h}$ for any  $v_h \in V_h^0$. 
Using this $\tau_h$ we prove an LBB condition for
$\blftwo$ on $\Velspacehdivfree$, which is the content of the next
lemma.  As $\tau_h \in \Stressspaceh$ has a zero trace, we cannot in
general control the divergence of a general $v_h \in V_h$ solely using
such a $\tau_h$. Therefore, to complete the proof of the full inf-sup
condition (in the proof of Theorem~\ref{th::lbbcondition} below), we utilize
an appropriate pressure test function as well.

\begin{lemma}
  \label{lem:intermediateLBB}
  For any nonzero $v_h \in V_h$ there exists a nonzero
  $\tau_h \in \Sigma_h$ satisfying 
  $
  b_2(\tau_h, v_h) \gtrsim \devnorm{v_h}^2$ and $ \| \tau_h \|_{\Sigma_h} 
  \lesssim \devnorm{v_h},
  $
  so  by~\eqref{eq:equivVh}, 
  \[
  \| v_h \|_{V_h} \lesssim   \sup_{\tau_h \in \Sigma_h} 
  \frac{b_2(\tau_h, v_h)}{ \| \tau_h \|_{\Sigma_h}} 
  \quad 
  \text{ for all } v_h \in V_h^0.
  \]
\end{lemma}
\begin{proof}
  Since the ideas are the same for $d=2$ and $3$, 
  for ease of exposition,   we give the details of the proof only 
  in the $d=2$ case.
  Because of the decomposition of the degrees of freedom into face and
  interior degrees of freedom (see \eqref{eq::facedofs} and \eqref{eq::eldofs}), we may decompose
  $\Sigma_h = \Sigma_h^0 \oplus \Sigma_h^1$ where
  $ \Sigma_h^0 = \oplus_{K \in \mesh} \mathcal{B}_k(T) $ and
  $\Sigma_h^1$ is the span of facet shape functions (see also Remark \ref{rem::basis}). 
  In particular,
  $\Sigma_h^1 $ contains the lowest order shape function $S^F$ with
  the property that $S^F_{nt} \in n_T^\perp$ and $||S^F_{nt}||_2 = 1$ on the facet $F$ and equals
  $(0,0)$ on all other facets in $\facets$. ($S^F$ can be explicity
  written down by mapping~\eqref{eq:refbasisface} or by appropriately
  scaling~\eqref{eq:basisface}). Given any $v_h \in V_h^0$, define 
  \begin{equation}
    \label{eq:19}
  \tau_h^0 := \sum_{T \in \mesh} \;\sum_{F \in \facetT} 
  -( S^F : \Dev{
  \nabla v_h}) \lambda_T^F S^F,
    \qquad
      \tau_h^1 := \sum_{F \in \facets} \frac 1 h 
  (\proj_F^0 \jump{(v_h)_t}) S^F,   \end{equation}
  where
  $\lambda_T^F$ is the barycentric coordinate of $T$
  that vanishes on
  $F$ (thus is $\lambda_T^FS^F$ is a linear inner $\normal\tangential$-bubble). Below we shall construct a linear combination of these
  functions to obtain the $\tau_h$ stated in the lemma.

  By~\eqref{eq:scale-s_L2norm} and~\eqref{eq:scale-s_nt}, a
  scaling argument (like in Lemma~\ref{lem::normeqsigma}) shows that 
  there is a mesh-independent $C_1$ such that 
  \begin{align} \label{eq:lbbtestfunctionnormestfac2}
  \big\| \tau_h^1 \big\|^2_{\Sigma_h}
    \le C_1 \sum\limits_{F \in \facets} 
    \frac{1}{h} 
    \big\| \proj_F^0 \jump{ (v_h)_t } \big\|_{F}^2
    \le 
    C_1\sum\limits_{F \in \facets} 
    \frac{1}{h} \left\| \jump{{ (v_h)_t}} \right\|_{F}^2.
  \end{align}
  A similar scaling argument also shows that 
  \begin{align} \label{eq:lbbtestfunctionnormesteltwo2}
    \Sigmanormh{ \tau_h^0 }^2
    \lesssim
    \sum\limits_{T \in \mesh}  \| \Dev{\nabla v_h} \|_{T}^2.
  \end{align}
  By construction, $(\tau_h^0)_{nt} = (0,0)$ and
  \begin{align*}
    b_2(\tau_h^0, v_h) = -\int_T \tau_h^0 : \nabla v_h \dx
    = 
    \int_T \sum_{F \in \facetT} 
    (S^F : \Dev{\nabla v_h})^2  \lambda_T^F.
  \end{align*}
  Since the functions $S^F$ form a basis for $\mathbb{D}$ 
  by Lemma~\ref{lem::normaltangentialprop}, 
  a scaling argument  shows that 
  \begin{align}\label{eq:lbbtestfunctionnormestel2}                                  
    b_2(\tau_h^0, v_h)
    & \gtrsim
     \sum\limits_{T \in \mesh}  \| \Dev{\nabla v_h} \|_{T}^2.
  \end{align}
  
  Next, set $\tau_h = \gamma_0 \tau_h^0 + \gamma_1 \tau_h^1$ where
  $\gamma_0$ and $\gamma_1$ are positive constants to be chosen. Then
  \begin{align*}
    b_2(\tau_h, v_h) 
    & \gtrsim
      \gamma_0 \sum\limits_{T \in \mesh}  \| \Dev{\nabla v_h} \|_{T}^2 + \gamma_1 b_2(\tau_h^1, v_h)
    && \text{by \eqref{eq:lbbtestfunctionnormestel2}}
    \\
    & =  
      \gamma_0 \sum\limits_{T \in \mesh}  \| \Dev{\nabla v_h} \|_{T}^2
      + \gamma_1
      \left(
      \sum_{T \in \mesh} -\int_T \tau_h^1: \nabla v_h \dx 
      + \sum_{F \in \facets} \int_F (\tau_h^1)_{nt}  \cdot \jump{(v_h)_t} \ds
      \right)
    \\
    & 
      = \gamma_0 \sum\limits_{T \in \mesh}  \| \Dev{\nabla v_h} \|_{T}^2
      - \gamma_1
      \sum_{T \in \mesh} \int_T \tau_h^1: \Dev{\nabla v_h} \dx 
      + \gamma_1 \sum_{F \in \facets} \frac 1 h 
      \big\| \proj_F^0 \jump{(v_h)_t} \big\|_{F}^2 \quad
    && \text{by~\eqref{eq:19}}.
  \end{align*}
  Applying the Cauchy Schwarz inequality and also Young's inequality with $\delta>0$ we further have
  \begin{align*}           
    b_2(\tau_h, v_h) 
    & 
      \gtrsim 
      \gamma_0 \sum\limits_{T \in \mesh}  \| \Dev{\nabla v_h} \|_{T}^2
      - \gamma_1
      \| \tau_h^1\|_{\Sigma_h} 
      \sqrt{ 
      \sum\limits_{T \in \mesh}  \| \Dev{\nabla v_h} \|_{T}^2
      }
      +
      \gamma_1\sum_{F \in \facets} \frac 1 h 
      \big\| \proj_F^0 \jump{(v_h)_t} \big\|_{F}^2
    \\
    & 
      \gtrsim
      \left(\gamma_0 - \frac{\gamma_1\delta
      }{2}\right) 
      \sum\limits_{T \in \mesh}  \| \Dev{\nabla v_h} \|_{T}^2 + 
      \left(1-\frac{C_1}{2 \delta
      }\right) \frac{\gamma_1}{h}
      \sum_{F \in \facets} \big\| \proj_F^0\jump{(v_h)_t} \big\|_{F}^2,       
  \end{align*}
  where in the last step we also used
  \eqref{eq:lbbtestfunctionnormestfac2}.
  Choosing
  $\delta = C_1$, $\gamma_1 = 1/\delta = 1/C_1,$ and $\gamma_0 =1$, 
  \begin{subequations}
    \label{eq:2cond4infsup}
    \begin{equation}
      \label{eq:23}
      b_2(\tau_h, v_h) 
      \gtrsim 
     \sum\limits_{T \in \mesh}  \| \Dev{\nabla v_h} \|_{T}^2 + 
      \sum_{F \in \facets} 
      \frac{1}{h}
      \big\| \proj_F^0\jump{(v_h)_t} \big\|_{F}^2.
    \end{equation}
    Let us also note that~\eqref{eq:lbbtestfunctionnormestfac2} and
    \eqref{eq:lbbtestfunctionnormesteltwo2} yield
    \begin{equation}
      \label{eq:22}
      \| \tau_h \|_{\Sigma_h} \lesssim 
      \sum\limits_{T \in \mesh}  \| \Dev{\nabla v_h} \|_{T}^2 + 
      \sum_{F \in \facets} 
      \frac{1}{h}
      \big\| \jump{(v_h)_t} \big\|_{F}^2.
    \end{equation}
  \end{subequations}
  The estimates~\eqref{eq:2cond4infsup} and the norm equivalences of
  \eqref{eq:equivVh} and Lemma \ref{lem::normeqvel} complete the proof.
\end{proof}

\begin{theorem}[Discrete LBB-condition] \label{th::lbbcondition}
  For all $v_h \in V_h$, 
  \begin{align} \label{eq:biglbb}
    \sup\limits_{(\tau_h,{q_h}) \in \Stressspaceh \times \Presspaceh}
    \frac{\blfone(v_h, {q_h}) + \blftwo(\tau_h,v_h)}
    {\Sigmanormh{ \tau_h }  + \Presnormh{ q_h} }
    \gtrsim \Velnormh{ v_h } 
  \end{align}
\end{theorem}
\begin{proof}
  By Lemma~\ref{lem:intermediateLBB}, for any $v_h \in V_h$, there is a
  $\tau_h \in \Sigma_h$ satisfying 
  $    \blftwo(\tau_h, v_h) \gtrsim \devnorm{v_h}^2$ and $\| \tau_h
  \|_{\Sigma_h} \lesssim \devnorm{v_h}$.
  Next we choose the pressure variable $q_h = \divergence (v_h)$, 
  which is possible
  due to the specific choice of $\Velspaceh$ and $\Presspaceh$, so
  that 
  $\blfone(v_h, {q}_h) =  \| \divergence (v_h) \|_{Q_h}^2.
  $
  With these choices of $\tau_h$ and $q_h$, we have 
  \begin{align*}
    \frac{\blfone(v_h, q_h) + \blftwo(\tau_h,v_h)}
    {\Sigmanormh{ \tau_h }  +  \Presnormh{ q_h }}
    \ge
    \frac{ \devnorm{v_h}^2 + \| \divergence(v_h) \|_{Q_h}^2 }
    {\Sigmanormh{ \tau_h }  +  \Presnormh{ q_h }}
    \gtrsim 
      \| v_h \|_{V_h}.
  \end{align*}
\end{proof}

\begin{remark}[Residual stabilization alternative] \label{rem::stabilization}
  A crucial ingredient in the proof of the LBB condition was the
  choice made in~\eqref{eq:19}. The choice of $\tau_h^0$ in terms of
  $(S^F: \dev{\nabla v_h}) \lambda_T^F S^F $ was admissible 
  as $\Dev{\nabla \Velvarh}$ is a polynomial of degree $k-1$ 
and $\Stressspaceh$ contains the element-wise bubbles of degree~$k$
in $\mathcal{B}_k(T)$. This choice would not be admissible if we had
used bubbles in $\mathcal{B}_{k-1}(T)$ instead of $\mathcal{B}_k(T)$.
Therefore, 
if we replace the stress
space by the lower degree space 
\[
\widetilde{\Sigma}_h := \{ \tau_h \in \Poly^{k-1}(\mesh,\rr^{d \times
                  d}):\;\trace{\tau_h} = 0,\;
                  \jump{(\tau_h)_{\normal\tangential}} =0\},
\]
the above proof can no longer be used to conclude stability of the resulting method.
Yet, its possible to get a good method (with optimal error convergence results) using $\widetilde{\Sigma}_h$
by a residual-based stabilization term. Define
 $c: \big[L^2(\om, \rr^{d \times d}) \times V\big] \times
\big[ L^2(\om, \rr^{d \times d}) \times V\big] \rightarrow \rr$ by 
 \begin{align*}
   c((\Stressvar, \Velvar), (\Stressvartest, \Velvartest)) := - \sum\limits_{T \in \mesh} \frac{\nu}{2} \int_T  (\frac{1}{\nu}\Stressvar - \nabla \Velvar ):(\frac{1}{\nu}\Stressvartest -\nabla \Velvartest) \dx.
 \end{align*}
When this form is added to the
system~\eqref{eq::discrmixedstressstokesweak} and $\Sigma_h$ is
replaced by $\widetilde{\Sigma}_h$, 
it is possible to prove stability.
\end{remark}

\begin{theorem}[Consistency] \label{th::consistency}
  The mass conserving mixed stress formulation
  \eqref{eq::discrmixedstressstokesweak} is consistent in the
  following sense. If the exact solution of the mixed Stokes
  problem~\eqref{eq::mixedstressstokes} is such that
  $\Velvar \in H^1(\om, \rr^d)$,
  $\Stressvar \in H^1(\om, \rr^{d \times d})$ and
  $\Presvar \in L^2_0(\om, \rr)$, then
  \begin{align*}
    \ablf(\Stressvar, \Stressvarhtest) + \blftwo(\Stressvarhtest, \Velvar) + \blftwo(\Stressvar, \Velvarhtest) + \blfone(\Velvarhtest, \Presvar) + \blfone(\Velvar, \Presvarhtest)   = (-\Forcevar, \Velvarhtest)_\om
  \end{align*}
  for all $\Velvarhtest \in \Velspaceh, \Presvarhtest \in \Presspaceh,
  $ and $\Stressvarhtest \in \Stressspaceh.$
\end{theorem}
\begin{proof}
  As the exact solutions $\Stressvar$ and $\Velvar$ are continuous we have $\jump{\Stressvar_{\normal\normal}}=0$ and $\jump{u_\tangential}=0$ on all faces $F \in \facets$ and thus using representations  \eqref{eq::blftwoequione} and \eqref{eq::blftwoequitwo} we have
  \begin{align*}
    \blftwo(\Stressvar, \Velvarhtest) = \sum\limits_{T \in \mesh} \int_T \divergence(\Stressvar) \cdot \Velvarhtest \dx - \sum\limits_{F \in \facets} \int_F \jump{\Stressvar_{nn}} (\Velvarhtest)_n \ds = \sum\limits_{T \in \mesh} \int_T \divergence(\Stressvar) \cdot \Velvarhtest \dx
  \end{align*}
  and
    \begin{align*}
    \blftwo(\Stressvarhtest, \Velvar) = -\sum\limits_{T \in \mesh} \int_T \Stressvarhtest : \nabla \Velvar \dx + \sum\limits_{F \in \facets} \int_F (\Stressvarhtest)_{nt} \cdot \jump{\Velvar_t} \ds = -\sum\limits_{T \in \mesh} \int_T \Stressvarhtest : \nabla \Velvar \dx . 
  \end{align*}
 Using $\divergence(\Velvar)=0$ we further get that $\blfone(\Velvar, \Presvarhtest) = 0$, so all together we have
   \begin{align*}
        \ablf(\Stressvar, \Stressvarhtest) &+ \blftwo(\Stressvarhtest, \Velvar) + \blftwo(\Stressvar, \Velvarhtest) + \blfone(\Velvarhtest, \Presvar) + \blfone(\Velvar, \Presvarhtest)  \\
                                       &= \int_\om \frac{1}{\nu} \Dev{\Stressvar} : \Dev{\Stressvarhtest} \dx -\sum\limits_{T \in \mesh} \int_T \Stressvarhtest : \nabla \Velvar \dx + \sum\limits_{T \in \mesh} \int_T \divergence(\Stressvar) \cdot \Velvarhtest \dx   + \int_\om \divergence(\Velvarhtest)\Presvar \dx 
   \end{align*}
   For the exact solution we have $\Dev{\Stressvar} = \nu \nabla \Velvar$. Further, as $\divergence(\Velvar)=0$, a simple calculation shows that $\Stressvarhtest : \nabla \Velvar= \Stressvarhtest : \Dev{\nabla \Velvar} = \Dev{\Stressvarhtest} : \nabla \Velvar$. Using integrating by parts for the last integral we conclude 
  \begin{align*}
        \ablf(\Stressvar, \Stressvarhtest) &+ \blftwo(\Stressvarhtest, \Velvar) + \blftwo(\Stressvar, \Velvarhtest) + \blfone(\Velvarhtest, \Presvar) + \blfone(\Velvar, \Presvarhtest)  \\
                                           &= \int_\om  \nabla \Velvar : \Dev{\Stressvarhtest} \dx -\sum\limits_{T \in \mesh} \int_T \Dev{\Stressvarhtest} : \nabla \Velvar \dx + \sum\limits_{T \in \mesh} \int_T \divergence(\Stressvar) \cdot \Velvarhtest \dx   + \int_\om \divergence(\Velvarhtest)\Presvar \dx  \\
    &= \int_\om \divergence(\Stressvar) \cdot \Velvarhtest \dx   +
      \int_\om \divergence(\Velvarhtest)\Presvar \dx  =\int_\om
      \big[ \divergence(\Stressvar) - \nabla \Presvar\big]
      \cdot \Velvarhtest \dx = \int_\om - \Forcevar \Velvarhtest \dx.
   \end{align*}
\end{proof}

\subsection{Error estimates}

\begin{theorem}[Optimal convergence rates] \label{th::optimalconvergence}
  Let $\Velvar \in H^1(\om, \rr^d) \cap H^m(\mesh, \rr^d)$, $\Stressvar \in  H^1(\om, \rr^{d \times d}) \cap H^{m-1}(\mesh, \rr^{d \times d})$ and $\Presvar \in  L^2_0(\om, \rr) \cap H^{m-1}(\mesh, \rr)$ be the exact solution of the mixed Stokes problem \eqref{eq::mixedstressstokes}. Further let $\Stressvarh$, $\Velvarh$ and $\Presvarh$ be the solution of the mass conserving mixed stress formulation \eqref{eq::discrmixedstressstokesweak}. For $s = \min(m-1,k)$ there holds
  \begin{align*}
    \Velnormh{ \Velvar - \Velvarh } + \frac{1}{\visc} \Sigmanormh{ \Stressvar - \Stressvarh} + \frac{1}{\visc} \Presnormh{\Presvar - \Presvarh} \lesssim h^s \left( ||u||_{H^{s+1}(\mesh)} + \frac{1}{\visc} ||\Stressvar||_{H^{s}(\mesh)} + \frac{1}{\visc} ||\Presvar||_{H^{s}(\mesh)} \right).
  \end{align*}  
\end{theorem}
\begin{proof}
  The proof is based on the discrete stability established above,
  which we shall use after bounding the error by triangle inequality
  into interpolation error and a discrete measure of error, as
  follows:
  \begin{equation}
    \label{eq:trgerrorsplit}
  \begin{aligned}
    \Velnormh{ \Velvar - \Velvarh } & + \frac{1}{\visc} \Sigmanormh{
      \Stressvar - \Stressvarh} +\frac{1}{\visc} \Presnormh{\Presvar
      - \Presvarh} 
      \\                                                                                            \lesssim
&    \;\Velnormh{ \Velvar - I_{\Velspaceh}\Velvar} + \frac{1}{\visc}\Sigmanormh{ \Stressvar - I_{\Stressspaceh}\Stressvar} + \frac{1}{\visc}\Presnormh{\Presvar - I_{\Presspaceh}\Presvar} \\
    + &\;   \Velnormh{ I_{\Velspaceh}\Velvar - \Velvarh } + \frac{1}{\visc}\Sigmanormh{ I_{\Stressspaceh}\Stressvar - \Stressvarh} + \frac{1}{\visc} \Presnormh{I_{\Presspaceh} \Presvar - \Presvarh}.
  \end{aligned}    
  \end{equation}
  Here $I_{\Stressspaceh}$ is the interpolation operator studied in
  Theorem~\ref{thm::interpolationop},  
  $I_{\Velspaceh}$ is the standard $H(\div)$-conforming
  interpolant -- see~\cite{brezzi1985two, RaviaThoma77} -- 
  and   $I_{\Presspaceh}$ is the 
$L^2$ projection into $\Presspaceh$. Note that for $s = \min(m-1,k)$ we have the approximation results
  \begin{align} \label{eq::approxvelpres}
    \Velnormh{ \Velvar - I_{\Velspaceh}\Velvar } \lesssim h^s ||u||_{H^{s+1}(\mesh)} \quad \textrm{and} \quad \Presnormh{\Presvar - I_{\Presspaceh}\Presvar} \lesssim h^s ||p||_{H^{s}(\mesh)}.
  \end{align}
  When this is combined with 
  \eqref{eq::approxsigma} of Theorem~\ref{thm::interpolationop},
  the first three terms  on the right hand side~\eqref{eq:trgerrorsplit} can be bounded as
  needed. 

  To bound  the remaining terms of~\eqref{eq:trgerrorsplit},  we first define the following norm on the product space $ \Velspaceh \times \Stressspaceh \times \Presspaceh$ given by
  \begin{align*}
    ||(\Velvarh,\Stressvarh, \Presvarh) ||_* := \sqrt{\visc} \Velnormh{ \Velvarh } + \frac{1}{\sqrt{\visc}} ( \Sigmanormh{\Stressvarh} + \Presnormh{\Presvarh}).
  \end{align*}
  Using the Brezzi theorem -- see for example in
  \cite{brezzi2012mixed} -- the LBB condition of the bilinear forms $\blfone$ and $\blftwo$ (Theorem \ref{th::lbbcondition}), the coercivity of $\ablf$ (Lemma \ref{th::coercivity}) and the continuity (Lemma \ref{th:continuity}) imply inf-sup stability of the
bilinear form
\begin{align*}
\blfbig(\Velvarh,\Stressvarh, \Presvarh;\Velvarhtest,\Stressvarhtest, \Presvarhtest) := \ablf(\Stressvarh, \Stressvarhtest) + \blfone(\Velvarh,\Presvarhtest) + \blfone(\Velvarhtest,\Presvarh) + \blftwo(\Stressvarh,\Velvarhtest) + \blftwo(\Stressvarhtest,\Velvarh), 
\end{align*}
with respect to the product space norm $||(\cdot,\cdot,\cdot) ||_*$, i.e., 
\begin{align*}
  ||(I_{\Velspaceh}\Velvar - \Velvarh, I_{\Stressspaceh}\Stressvar - \Stressvarh, I_{\Presspaceh} \Presvar - \Presvarh) ||_* 
& \le \sup\limits_{(\Velvarhtest, \Stressvarhtest, \Presvarhtest) \in \Velspaceh \times \Stressspaceh \times \Presspaceh } \frac{ \blfbig(I_{\Velspaceh}\Velvar - \Velvarh,I_{\Stressspaceh}\Stressvar - \Stressvarh, I_{\Presspaceh} \Presvar - \Presvarh;\Velvarhtest,\Stressvarhtest, \Presvarhtest)}{||(\Velvarhtest,\Stressvarhtest, \Presvarhtest) ||_*}
\\& \le
 \sup\limits_{(\Velvarhtest, \Stressvarhtest, \Presvarhtest) \in \Velspaceh \times \Stressspaceh \times \Presspaceh } \frac{ \blfbig(I_{\Velspaceh}\Velvar - \Velvar,I_{\Stressspaceh}\Stressvar - \Stressvar, I_{\Presspaceh} \Presvar - \Presvar;\Velvarhtest,\Stressvarhtest, \Presvarhtest)}{||(\Velvarhtest,\Stressvarhtest, \Presvarhtest) ||_*},
\end{align*}
where we used the consistency result of Theorem~\ref{th::consistency} in the last step.

Next, we estimate the terms that form 
$\blfbig(I_{\Velspaceh}\Velvar - \Velvar,I_{\Stressspaceh}\Stressvar -
\Stressvar, I_{\Presspaceh} \Presvar -
\Presvar;\Velvarhtest,\Stressvarhtest, \Presvarhtest)$.
Using the Cauchy Schwarz inequality,
\begin{align*}
  &\ablf(I_{\Stressspaceh}\Stressvar - \Stressvar, \Stressvarhtest) + \blfone(I_{\Velspaceh}\Velvar - \Velvar,\Presvarhtest ) + \blfone(\Velvarhtest,I_{\Presspaceh} \Presvar - \Presvar) \\
  &\lesssim (\frac{1}{\sqrt{\visc}} \Sigmanormh{I_{\Stressspaceh}\Stressvar - \Stressvar} \frac{1}{\sqrt{\visc}} \Sigmanormh{\Stressvarhtest} ) + (\sqrt{\visc}\Velnormh{I_{\Velspaceh}\Velvar - \Velvar}\frac{1}{\sqrt{\visc}} \Presnormh{\Presvarhtest}) + (\sqrt{\visc}\Velnormh{\Velvartest}\frac{1}{\sqrt{\visc}} \Presnormh{I_{\Presspaceh} \Presvar - \Presvar}) \\
  &\le  ||(I_{\Velspaceh}\Velvar - \Velvar, I_{\Stressspaceh}\Stressvar - \Stressvar, I_{\Presspaceh} \Presvar - \Presvar) ||_* ||(\Velvarhtest,\Stressvarhtest, \Presvarhtest) ||_*
\end{align*}
For the terms including the bilinear form $\blftwo$ we also have by the Cauchy Schwarz inequality applied on each element and each facet
\begin{align*}
  \blftwo(I_{\Stressspaceh}\Stressvar - \Stressvar,\Velvarhtest ) &+ \blftwo(\Stressvarhtest, I_{\Velspaceh}\Velvar - \Velvar ) \\
                                                                  &
                                                                    \lesssim \sum\limits_{F \in \facets} \sqrt{h} ||(I_{\Stressspaceh}\Stressvar - \Stressvar)_{\normal\tangential}||_{F} \frac{1}{\sqrt{h}} ||\jump{(\Velvarhtest)_{\tangential}}||_{F} + \sum\limits_{T \in \mesh} || I_{\Stressspaceh}\Stressvar - \Stressvar||_{T} ||\nabla \Velvarhtest ||_{T}\\
  &+ \sum\limits_{F \in \facets} \sqrt{h} ||(\Stressvarhtest)_{\normal\tangential}||_{F} \frac{1}{\sqrt{h}} ||\jump{(I_{\Velspaceh}\Velvar - \Velvar)_{\tangential}}||_{F} + \sum\limits_{T \in \mesh} ||\Stressvarhtest||_{T} ||\nabla (I_{\Velspaceh}\Velvar - \Velvar)||_{T}.
\end{align*}
Scaling with $\sqrt{\visc}$ and applying  the norm equivalence Lemma~\ref{lem::normeqsigma} finally yields
\begin{align*}
  \blftwo(I_{\Stressspaceh}\Stressvar &- \Stressvar,\Velvarhtest ) + \blftwo(\Stressvarhtest, I_{\Velspaceh}\Velvar - \Velvar ) \lesssim \\
  &\left(\frac{1}{\sqrt{\visc}}( \Sigmanormh{I_{\Stressspaceh}\Stressvar - \Stressvar} + \frac{1}{\sqrt{\visc}} \sqrt{\sum\limits_{F \in \facets} h || (I_{\Stressspaceh}\Stressvar - \Stressvar)_{\normal\tangential} ||^2_{F}} + \sqrt{\visc} \Velnormh{I_{\Velspaceh}\Velvar - \Velvar} \right) ||(\Velvarhtest,\Stressvarhtest, 0) ||_* .
\end{align*}
All together this leads to the estimate
\begin{align*}
  ||(I_{\Velspaceh}\Velvar - \Velvarh,& I_{\Stressspaceh}\Stressvar - \Stressvarh, I_{\Presspaceh} \Presvar - \Presvarh) ||_*  \\
  &\lesssim ||(I_{\Velspaceh}\Velvar - \Velvar,I_{\Stressspaceh}\Stressvar - \Stressvar, I_{\Presspaceh} \Presvar - \Presvar) ||_* + \frac{1}{\sqrt{\visc}} \sqrt{\sum\limits_{F \in \facets} h || (I_{\Stressspaceh}\Stressvar - \Stressvar)_{\normal\tangential} ||^2_{F}}.
\end{align*}
Again, with \eqref{eq::approxvelpres} and \eqref{eq::approxsigma} we conclude the proof.
\end{proof}

\subsection{Pressure robustness}

We define the continuous Helmholtz projector $\mathbb{P}$ as the rotational part of a Helmholtz decomposition (see \cite{girault2012finite}) of a given load $\Forcevar$
\begin{align*}
  f = \nabla \theta + \xi =: \nabla \theta + \mathbb{P}(\Forcevar),
\end{align*}
with $\theta \in H^1(\om)/\rr$ and  $\xi =: \mathbb{P}(f) \in \{ \Velvartest \in H_0(\divergence, \om): \divergence(\Velvartest) = 0\}$.  Testing the second line of \eqref{eq::mixedstressstokesweak} with an arbitrary divergence free testfunction $\Velvartest \in \{ \Velvartest \in H_0(\divergence, \om): \divergence(\Velvartest) = 0\}$, we see that
\begin{align*} 
    \ip{\divergence \Stressvar, \Velvartest}_{H_0(\div, \om)} = 
    -\ip{ \mathbb{P}(\Forcevar), \Velvartest}_{H_0(\div, \om)},
\end{align*}
hence  $\Stressvar = \nu \nabla \Velvar$ is steered only by
a part of $f$, namely $\mathbb{P}(\Forcevar)$. If the right hand side
is perturbed by a gradient field $\nabla \alpha$, then $\sigma$ and
$u$ should not change as $\mathbb{P}(\Forcevar + \nabla \alpha) =
\mathbb{P}(\Forcevar)$. In the work by \cite{linke2014role} this
relation was discussed in a discrete setting. If a discrete method
fulfills this property, it is called pressure robust because one can
then deduce an $H^1$-velocity error that is independent of the
pressure. The convergence estimate of
Theorem~\ref{th::optimalconvergence} includes the scaled term
${1}/{\visc} ||\Presvar||_{H^{s}(\mesh)}$ which blows up as  
 $\nu \rightarrow 0.$ However, the mass conserving mixed stress
 formulation \eqref{eq::discrmixedstressstokesweak} is pressure
 robust, allowing us to conclude
that velocity errors do not blow up as $\visc \to 0$
by virtue of  the next theorem.

\begin{theorem}[Pressure robustness] \label{th::pressurerobust}
  Let $\Velvar \in H^1(\om, \rr^d) \cap H^m(\mesh, \rr^d)$ and let  $\Stressvar \in  H^1(\om, \rr^{d \times d}) \cap H^{m-1}(\mesh, \rr^{d \times d})$ be the exact solution of the mixed Stokes problem \eqref{eq::mixedstressstokes}. Further let $\Stressvarh$, $\Velvarh$ be the solution of the mass conserving mixed stress formulation \eqref{eq::discrmixedstressstokesweak}. For $s = \min(m-1,k)$ there holds
  \begin{align*}
    \Velnormh{ \Velvar - \Velvarh } + \frac{1}{\visc}\Sigmanormh{ \Stressvar - \Stressvarh} \lesssim h^s ||u||_{H^{s+1}(\mesh)}.
  \end{align*}  
\end{theorem}
\begin{proof}
  The proof follows along the lines of the proof of Theorem~\ref{th::optimalconvergence}. Using the triangle inequality,
  \begin{align*}
    \Velnormh{ \Velvar - \Velvarh } +\frac{1}{\visc} \Sigmanormh{ \Stressvar - \Stressvarh} 
                                                                                            \lesssim    \Velnormh{ \Velvar - I_{\Velspaceh}\Velvar} + \frac{1}{\visc}\Sigmanormh{ \Stressvar - I_{\Stressspaceh}\Stressvar}  
    +   \Velnormh{ I_{\Velspaceh}\Velvar - \Velvarh } + \frac{1}{\visc}\Sigmanormh{ I_{\Stressspaceh}\Stressvar - \Stressvarh}.
  \end{align*}
  The first two terms can be estimated using the approximation
  results \eqref{eq::approxvelpres} and \eqref{eq::approxsigma}. Next
  note that from the LBB condition of Lemma~\ref{lem:intermediateLBB}
  on $V_h^0$ and the trivial coercivity inequality
  $ \ablf(\Stressvarh,\Stressvarh) \geq (1/\visc) \Sigmanormh{
    \Stressvarh }^2$
  for all $\Stressvarh \in \Stressspaceh$, we conclude
  inf-sup stability of the
  bilinear form
  $\blfbig(\Velvarh,\Stressvarh, 0;\Velvarhtest,\Stressvarhtest, 0) $
  with respect to the product space norm $||(\cdot,\cdot,0) ||_*$ on
  the subspace $\Velspacehdivfree \times \Stressspaceh \times \{0\}$,
    i.e.,
\begin{align*}
  \Velnormh{ I_{\Velspaceh}\Velvar - \Velvarh } + \frac{1}{\visc}\Sigmanormh{ I_{\Stressspaceh}\Stressvar - \Stressvarh} &= \frac{1}{\sqrt{\visc}}  ||(I_{\Velspaceh}\Velvar - \Velvarh, I_{\Stressspaceh}\Stressvar - \Stressvarh,0) ||_* \\
  &\le \sup\limits_{(\Velvarhtest, \Stressvarhtest) \in \Velspacehdivfree \times \Stressspaceh} \frac{ \blfbig(I_{\Velspaceh}\Velvar - \Velvarh,I_{\Stressspaceh}\Stressvar - \Stressvarh, 0;\Velvarhtest,\Stressvarhtest, 0)}{\sqrt{\visc}||(\Velvarhtest,\Stressvarhtest, 0) ||_*}.
\end{align*}
Note that the form  is continuous by Lemma~\ref{th:continuity}.
By steps similar to those in the proof of the consistency result of Theorem \ref{th::consistency} we have 
\begin{align*}
\blfbig(\Velvar,\Stressvar, 0;\Velvarhtest,\Stressvarhtest, 0) = \int_\om \divergence(\Stressvar) \cdot \Velvarhtest = \int_\om -f \cdot \Velvarhtest + \int_\om \nabla \Presvar \cdot \Velvarhtest = \int_\om -f \cdot \Velvarhtest 
\end{align*}
for all $ \Velvarhtest,\Stressvarhtest \in \Velspacehdivfree \times \Stressspaceh$,
where we used   $ \divergence(\Stressvar) = -f + \nabla \Presvar$ and
integration by parts for $\nabla \Presvar$. This shows that the method
is also consistent on the subspace of divergence-free velocity test
functions, a key ingredient to obtain pressure robustness. We now have
\begin{align*}
\sup\limits_{(\Velvarhtest, \Stressvarhtest) \in \Velspacehdivfree \times \Stressspaceh} \frac{ \blfbig(I_{\Velspaceh}\Velvar - \Velvarh,I_{\Stressspaceh}\Stressvar - \Stressvarh, 0;\Velvarhtest,\Stressvarhtest, 0)}{\sqrt{\visc}||(\Velvarhtest,\Stressvarhtest, \Presvarhtest) ||_*} = \sup\limits_{(\Velvarhtest, \Stressvarhtest) \in \Velspacehdivfree \times \Stressspaceh} \frac{ \blfbig(I_{\Velspaceh}\Velvar - \Velvar,I_{\Stressspaceh}\Stressvar - \Stressvar, 0;\Velvarhtest,\Stressvarhtest, 0)}{\sqrt{\visc}||(\Velvarhtest,\Stressvarhtest, \Presvarhtest) ||_*}.
\end{align*}
The rest of the proof follows along the previous lines using the
identity $\Stressvar = \visc \nabla \Velvar$ and we obtain 
\begin{align*}
\Velnormh{ \Velvar - \Velvarh } + \frac{1}{\visc}\Sigmanormh{ \Stressvar - \Stressvarh} \lesssim h^s (||u||_{H^{s+1}(\mesh)} + \frac{1}{\visc} ||\Stressvar||_{H^{s}(\mesh)}) \le   h^s ||u||_{H^{s+1}(\mesh)}.
\end{align*}
\end{proof}

\section{Numerical examples} \label{sec:numerics}

In the following we present a numerical example to validate the results of Section \ref{sec::apriorianalysis}. All numerical examples were implemented within the finite element library NGSolve/Netgen, see \cite{netgen, ngsolve}. Let $\om = [0,1]^d$ and choose the right hand side $f = -\div(\Stressvar) + \nabla \Presvar$ with the exact solution given by
\begin{align*}
  \Stressvar &= \visc\nabla\curl(\psi_2),\quad \textrm{and} \quad  \Presvar := x^5 + y^5 - \frac{1}{3} \quad \textrm{for } d=2\\
  \Stressvar &= \visc\nabla\curl(\psi_3,\psi_3,\psi_3), \quad \textrm{and} \quad  \Presvar := x^5 + y^5 +z^5 - \frac{1}{2} \quad \textrm{for } d=3,
\end{align*}
where $\psi_2 :=x^2(x-1)^2y^2(y-1)^2$ and $\psi_3 :=
x^2(x-1)^2y^2(y-1)^2z^2(z-1)^2$ defines velocity through a vector and
scalar potential in  two and three dimensions respectively. 
In Figure~\ref{fig:convergenceplot} different errors are plotted for varying polynomial orders $k = 2,3,4,5$ in the two dimensional case with a fixed viscosity $\nu = 10^{-3}$. As predicted by Theorem~\ref{th::optimalconvergence}, the $H^1$-seminorm error of the velocity, the $L^2$-norm error of the stress and the $L^2$-norm error of the pressure have the same optimal convergence rate.

\begin{figure}[h]
  \begin{center}
    \pgfplotstableread{convergence/convergence_2_2d.out} \convtwo
\pgfplotstableread{convergence/convergence_3_2d.out} \convthree
\pgfplotstableread{convergence/convergence_4_2d.out} \convfour
\pgfplotstableread{convergence/convergence_5_2d.out} \convfive

\definecolor{myblue}{RGB}{62,146,255}
\definecolor{mygreen}{RGB}{22,135,118}
\definecolor{myred}{RGB}{255,145,0}

\begin{tikzpicture}
  [  
  scale=1
  ]
  \begin{axis}[
    name=plot1,
    scale=0.8,
    title = $|| \nabla \Velvar - \nabla \Velvarh ||_0$,
    legend style={text height=0.7em },
    legend style={draw=none},
    style={column sep=0.1cm},
    xlabel=$|\mesh|$,
    x label style={at={(0.85,0.04)},anchor=west},
    xmode=log,
    ymode=log,
    ytick = {1e-10,1e-8,1e-6,1e-4,1e-2,1e-0},
    y tick label style={
      /pgf/number format/.cd,
      fixed,
      precision=2
    },
    xtick = {1e1,1e2,1e3,1e4},
    x tick label style={
      /pgf/number format/.cd,
      fixed,
      precision=2
    },
    grid=both,
    legend style={
      cells={align=left},
      anchor = north west
    },
    ]
    \addplot[line width=0.5pt, color=black] table[x=0,y expr={1/(sqrt{\thisrow{0}}*sqrt{\thisrow{0}})}]{\convtwo};
    \addplot[line width=0.5pt, densely dotted, color=black] table[x=0,y expr={1/(sqrt{\thisrow{0}}*sqrt{\thisrow{0}}*sqrt{\thisrow{0}})}]{\convtwo};
    \addplot[line width=0.5pt, dashed, color=black] table[x=0,y expr={1/(sqrt{\thisrow{0}}*sqrt{\thisrow{0}}*sqrt{\thisrow{0}}*sqrt{\thisrow{0}})}]{\convtwo};
    \addplot[line width=0.5pt, dashdotted, color=black] table[x=0,y expr={1/(sqrt{\thisrow{0}}*sqrt{\thisrow{0}}*sqrt{\thisrow{0}}*sqrt{\thisrow{0}}*sqrt{\thisrow{0}})}]{\convtwo};
    
    \addplot[line width=0.5pt, color=magenta, mark=*] table[x=0, y=1]{\convtwo};
    \addplot[line width=0.5pt, color=myblue, mark=diamond*] table[x=0, y=1]{\convthree};    
    \addplot[line width=0.5pt, color=myred, mark=triangle*] table[x=0, y=1]{\convfour};   
    \addplot[line width=0.5pt, color=mygreen, mark=square*] table[x=0, y=1]{\convfive};
  \end{axis}
   \begin{axis}[
     name=plot2,
     at=(plot1.right of south east),
     anchor=left of south west,
     title = $|| \Stressvar - \Stressvarh ||_0$,
    scale=0.8,
    legend style={text height=0.7em },
    legend style={draw=none},
    style={column sep=0.15cm},
    xlabel=$|\mesh|$,
    x label style={at={(0.85,0.04)},anchor=west},
    xmode=log,
    ymode=log,
    ytick = {1e-10,1e-8,1e-6,1e-4,1e-2,1e-0},
    y tick label style={
      /pgf/number format/.cd,
      fixed,
      precision=2
    },
    xtick = {1e1,1e2,1e3,1e4},
    x tick label style={
      /pgf/number format/.cd,
      fixed,
      precision=2
    },    
    grid=both,
    legend style={
      cells={align=left},
      at={(1.05,1)},
      anchor = north west
    },
    ]
    \addplot[line width=0.5pt, color=black] table[x=0,y expr={1/(4*sqrt{\thisrow{0}}*sqrt{\thisrow{0}})}]{\convtwo};
    \addplot[line width=0.5pt, densely dotted, color=black] table[x=0,y expr={1/(4*sqrt{\thisrow{0}}*sqrt{\thisrow{0}}*sqrt{\thisrow{0}})}]{\convtwo};
    \addplot[line width=0.5pt, dashed, color=black] table[x=0,y expr={1/(10*sqrt{\thisrow{0}}*sqrt{\thisrow{0}}*sqrt{\thisrow{0}}*sqrt{\thisrow{0}})}]{\convtwo};
    \addplot[line width=0.5pt, dashdotted, color=black] table[x=0,y expr={1/(10*sqrt{\thisrow{0}}*sqrt{\thisrow{0}}*sqrt{\thisrow{0}}*sqrt{\thisrow{0}}*sqrt{\thisrow{0}})}]{\convtwo};
    
    \addplot[line width=0.5pt, color=magenta, mark=*] table[x=0, y=4]{\convtwo};   
    \addplot[line width=0.5pt, color=myblue, mark=diamond*] table[x=0, y=4]{\convthree};   
    \addplot[line width=0.5pt, color=myred, mark=triangle*] table[x=0, y=4]{\convfour};    
    \addplot[line width=0.5pt, color=mygreen, mark=square*] table[x=0, y=4]{\convfive};

  \end{axis}
\end{tikzpicture}
\begin{tikzpicture}
  [  
  scale=1
  ]
  \begin{axis}[
    name=plot1,
    scale=0.8,
    title = $|| \Presvar - \Presvarh ||_0$,
    legend style={text height=0.7em },
    legend style={draw=none},
    style={column sep=0.1cm},
    xlabel=$|\mesh|$,
    x label style={at={(0.85,0.04)},anchor=west},
    xmode=log,
    ymode=log,
    ytick = {1e-10,1e-8,1e-6,1e-4,1e-2,1e-0},
    y tick label style={
      /pgf/number format/.cd,
      fixed,
      precision=2
    },
    xtick = {1e1,1e2,1e3,1e4},
    x tick label style={
      /pgf/number format/.cd,
      fixed,
      precision=2
    },
    grid=both,
    legend style={
      cells={align=left},
      anchor = north west
    },
    ]
    \addplot[line width=0.5pt, color=black] table[x=0,y expr={3/(sqrt{\thisrow{0}}*sqrt{\thisrow{0}})}]{\convtwo};
    \addplot[line width=0.5pt, densely dotted, color=black] table[x=0,y expr={1/(sqrt{\thisrow{0}}*sqrt{\thisrow{0}}*sqrt{\thisrow{0}})}]{\convtwo};
    \addplot[line width=0.5pt, dashed, color=black] table[x=0,y expr={1/(4*sqrt{\thisrow{0}}*sqrt{\thisrow{0}}*sqrt{\thisrow{0}}*sqrt{\thisrow{0}})}]{\convtwo};
    \addplot[line width=0.5pt, dashdotted, color=black] table[x=0,y expr={1/(5*sqrt{\thisrow{0}}*sqrt{\thisrow{0}}*sqrt{\thisrow{0}}*sqrt{\thisrow{0}}*sqrt{\thisrow{0}})}]{\convtwo};
    
    \addplot[line width=0.5pt, color=magenta, mark=*] table[x=0, y=3]{\convtwo};
    \addplot[line width=0.5pt, color=myblue, mark=diamond*] table[x=0, y=3]{\convthree};    
    \addplot[line width=0.5pt, color=myred, mark=triangle*] table[x=0, y=3]{\convfour};   
    \addplot[line width=0.5pt, color=mygreen, mark=square*] table[x=0, y=3]{\convfive};
  \end{axis}
   \begin{axis}[
     name=plot2,
     at=(plot1.right of south east),
     anchor=left of south west,
     title = $|| \Velvar - \Velvarh ||_0$,     
    scale=0.8,
    legend style={text height=0.7em },
    legend style={draw=none},
    style={column sep=0.15cm},
    xlabel=$|\mesh|$,
    x label style={at={(0.85,0.04)},anchor=west},
    xmode=log,
    ymode=log,
    ytick = {1e-12,1e-10,1e-8,1e-6,1e-4,1e-2,1e-0},
    y tick label style={
      /pgf/number format/.cd,
      fixed,
      precision=2
    },
    xtick = {1e1,1e2,1e3,1e4},
    x tick label style={
      /pgf/number format/.cd,
      fixed,
      precision=2
    },    
    grid=both,
    legend style={
      cells={align=left},
      at={(1.05,1)},
      anchor = north west
    },
    ]
    
    \addplot[line width=0.5pt, color=black] table[x=0,y expr={1/(4*sqrt{\thisrow{0}}*sqrt{\thisrow{0}})}]{\convtwo};
    \addplot[line width=0.5pt, densely dotted, color=black] table[x=0,y expr={1/(4*sqrt{\thisrow{0}}*sqrt{\thisrow{0}}*sqrt{\thisrow{0}})}]{\convtwo};
    \addplot[line width=0.5pt, dashed, color=black] table[x=0,y expr={1/(10*sqrt{\thisrow{0}}*sqrt{\thisrow{0}}*sqrt{\thisrow{0}}*sqrt{\thisrow{0}})}]{\convtwo};
    \addplot[line width=0.5pt, dashdotted, color=black] table[x=0,y expr={1/(10*sqrt{\thisrow{0}}*sqrt{\thisrow{0}}*sqrt{\thisrow{0}}*sqrt{\thisrow{0}}*sqrt{\thisrow{0}})}]{\convtwo};
    \addplot[line width=0.5pt, loosely dotted, color=black] table[x=0,y expr={1/(10*sqrt{\thisrow{0}}*sqrt{\thisrow{0}}*sqrt{\thisrow{0}}*sqrt{\thisrow{0}}*sqrt{\thisrow{0}}*sqrt{\thisrow{0}})}]{\convtwo};
    
    \addplot[line width=0.5pt, color=magenta, mark=*] table[x=0, y=2]{\convtwo};   
    \addplot[line width=0.5pt, color=myblue, mark=diamond*] table[x=0, y=2]{\convthree};   
    \addplot[line width=0.5pt, color=myred, mark=triangle*] table[x=0, y=2]{\convfour};    
    \addplot[line width=0.5pt, color=mygreen, mark=square*] table[x=0, y=2]{\convfive};

  \end{axis}
\end{tikzpicture}
\begin{tikzpicture}
  \begin{customlegend}[hide axis,
    xmin=10,
    xmax=50,
    ymin=0,
    ymax=0.4,
    legend columns=9,
    legend style={draw=none},
    legend entries={$h^2$,$h^3$,$h^4$,$h^5$,$h^6$,$k=2$,$k=3$,$k=4$,$k=5$}]
   
    \addlegendimage{black}
    \addlegendimage{black,densely dotted}
    \addlegendimage{black,dashed}
    \addlegendimage{black,dashdotted}
    \addlegendimage{black,loosely dotted}
    
    \addlegendimage{magenta,mark=*}
    \addlegendimage{myblue,mark=diamond*}
    \addlegendimage{myred,mark=triangle*}
    \addlegendimage{mygreen,mark=square*}
  \end{customlegend}
\end{tikzpicture}
  \end{center}
  \vspace{-0.7cm}
  \caption{Convergence plots for the two dimensional case with a fixed viscosity $\nu = 10^{-3}$.}
\label{fig:convergenceplot}
\end{figure}
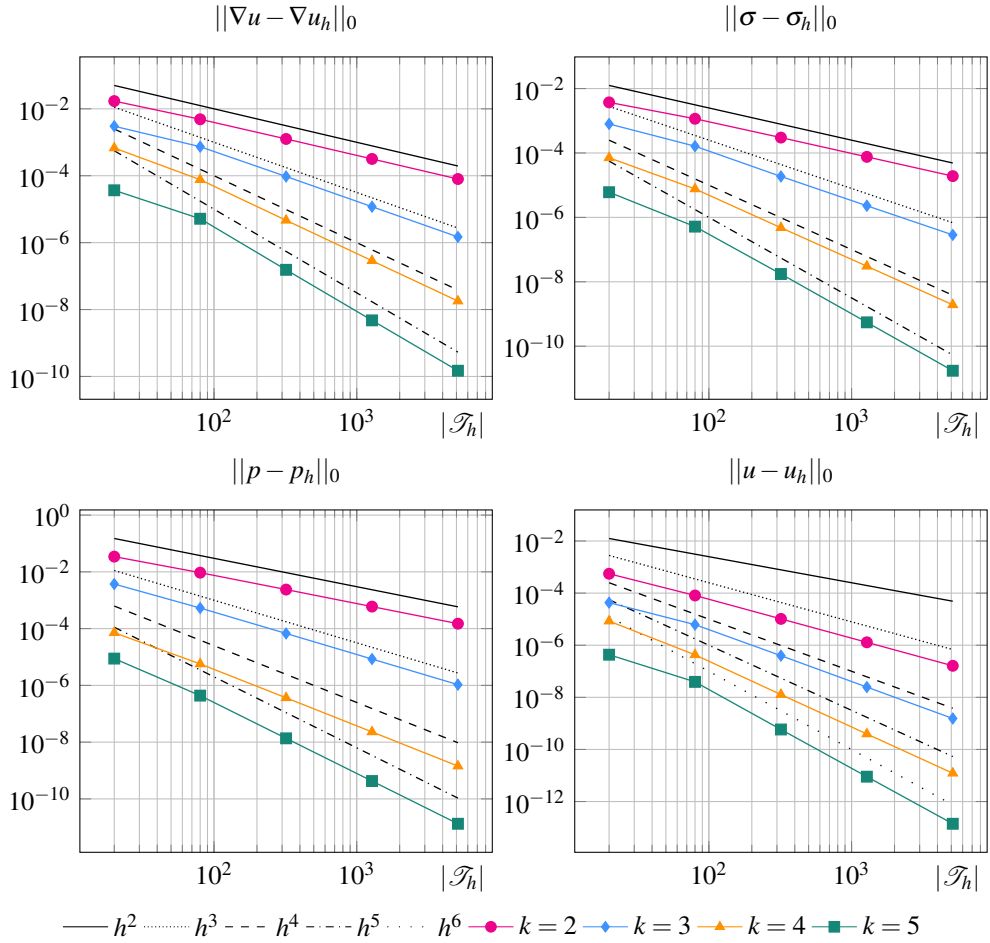

The $L^2$-norm of the velocity error converges at one higher order as
shown in  the bottom right plot of Figure~\ref{fig:convergenceplot}.
This can be explained by the standard Aubin-Nitsche duality argument,
by which we can prove
\begin{align*}
  || \Velvar - \Velvarh ||_{L^2(\om)} \le h^{k+1} || u ||_{H^{k+1}(\mesh)}
\end{align*}
whenever  the problem admits full elliptic regularity and the exact solution
$u$ is smoother.  This argument works in both two and three dimensions. The higher observed rate of convergence in three dimensions (for
$\visc = 10^{-3}$), given by the estimated order of convergence (eoc), can be seen in Table \ref{threedexample}.


Next, we study pressure robustness. The above-mentioned right hand
side $f$ consists of an irrotational part (the gradient of the
pressure) and a part with curl. We study how the velocity error (in
$H^1$ seminorm) varies as $\visc \to 0$ for the presented MCS method
and the standard Taylor-Hood method -- see e.g., \cite{brezzi:falk}
and \cite{girault2012finite} -- using the same polynomial approximation
order for the velocity in the two dimensional case.  We observe in
Figure~\ref{fig:probustplot} that the error of the Taylor-Hood method
increases as $\visc \to 0$ and behaves as if it were scaled by a
factor $1/\nu$ for small values of $\nu$. This is the locking
phenomenon we discussed earlier: clearly the Taylor-Hood method is not
pressure robust (and does not provide exactly divergence-free
numerical velocity).  In contrast, the velocity errors in the MCS
method (also in Figure~\ref{fig:probustplot}) appear to be not
influenced by varying values of $\visc$. This behaviour is observed for
several polynomial orders $k=2,3,4$.  These observations match the
predictions of Theorem~\ref{th::pressurerobust}.

\begin{table}[h]
\begin{center}
\footnotesize
\begin{tabular}{@{~}c@{~}|@{~~~}c@{~~~(}c@{)~~~}c@{~~~(}c@{)~~~}c@{~~~(}c@{)~~~}c@{~~~(}c@{)}}
      \toprule
  $|\mathcal{T}|$ & $|| \nabla \Velvar - \nabla \Velvarh||_0$ & \footnotesize eoc &$|| \Stressvar - \Stressvarh||_0$ & \footnotesize eoc &$|| \Presvar - \Presvarh||_0$ & \footnotesize eoc &$|| \Velvar - \Velvarh||_0$ & \footnotesize eoc  \\
\midrule
    \multicolumn{9}{c}{$k=1$}\\
28& \num{0.004571830455453789}&--& \num{0.003505660174751253}&--& \num{0.2445663237697641}&--& \num{0.0004303199553831272}&--\\
224& \num{0.0038957718956525856}&\numeoc{0.23086269522028346}& \num{0.0026998956026624806}&\numeoc{0.37678253014048}& \num{0.1667671971580337}&\numeoc{0.5523902208214316}& \num{0.00026052069950875895}&\numeoc{0.7240117389697258}\\
1792& \num{0.0022610459009523225}&\numeoc{0.7849189273483329}& \num{0.001252560538828451}&\numeoc{1.108023289886233}& \num{0.08940907149326324}&\numeoc{0.8993424202973235}& \num{7.572024553946509e-05}&\numeoc{1.7826470110711856}\\
14336& \num{0.001128456944331142}&\numeoc{1.0026389058687597}& \num{0.0006259234563641155}&\numeoc{1.000822186392283}& \num{0.04605099642700857}&\numeoc{0.9571888425386336}& \num{1.945844426826144e-05}&\numeoc{1.9602827202012476}\\
114688& \num{0.0005637272276863808}&\numeoc{1.0012822187468833}& \num{0.00031183189852873497}&\numeoc{1.0052177274060965}& \num{0.023221776996557177}&\numeoc{0.9877539973596987}& \num{4.8947644110162155e-06}&\numeoc{1.9910850405729392}\\
    \midrule
    \multicolumn{9}{c}{$k=2$}\\
28& \num{0.002809695156178579}&--& \num{0.0018675901590948102}&--& \num{0.0745049307161993}&--& \num{0.00013681043155242657}&--\\   
224& \num{0.0014806755334779444}&\numeoc{0.9241580791255488}& \num{0.00046162629731745036}&\numeoc{2.016380574726804}& \num{0.031059445732618422}&\numeoc{1.2623038213781705}& \num{3.4430807013592997e-05}&\numeoc{1.990406336029456}\\                                  
1792& \num{0.0005761812245964354}&\numeoc{1.3616609767497394}& \num{0.0001763508031182975}&\numeoc{1.3882772659215066}& \num{0.009518865088896598}&\numeoc{1.706170604641592}& \num{8.188122822394921e-06}&\numeoc{2.0720953480510516}\\                                  
14336& \num{0.00017009293038200445}&\numeoc{1.7601994705705861}& \num{4.8506169024657504e-05}&\numeoc{1.8622079996131087}& \num{0.0025334178625218114}&\numeoc{1.9097045196430396}& \num{1.2529196440918618e-06}&\numeoc{2.708238852558938}\\                             
114688& \num{4.3844935482378406e-05}&\numeoc{1.9558410670670516}& \num{1.2504180011133977e-05}&\numeoc{1.955757788475428}& \num{0.0006437098670247017}&\numeoc{1.976602566830205}& \num{1.6230950603503734e-07}&\numeoc{2.9484744882125358}\\ 
    \midrule
    \multicolumn{9}{c}{$k=3$}\\
28& \num{0.0010039158371449477}&--& \num{0.0002865551141187956}&--& \num{0.006749483419566198}&--& \num{2.549820683872712e-05}&--\\ 
224& \num{0.00047546929046932146}&\numeoc{1.0782142579736644}& \num{9.166622863944649e-05}&\numeoc{1.6443504200971415}& \num{0.001550192898021939}&\numeoc{2.1223293396479233}& \num{6.290461705997669e-06}&\numeoc{2.019157976396193}\\                                  
1792& \num{0.00014758939976579767}&\numeoc{1.6877630563418733}& \num{1.7027566063140106e-05}&\numeoc{2.428518089378506}& \num{0.0002619025070816613}&\numeoc{2.565345973036723}& \num{1.0444193618927235e-06}&\numeoc{2.590464803311981}\\                                
14336& \num{2.0080795256093505e-05}&\numeoc{2.8777007968851795}& \num{2.385215567393521e-06}&\numeoc{2.835680666581789}& \num{3.53049538620429e-05}&\numeoc{2.8910873335564626}& \num{7.194022880602107e-08}&\numeoc{3.8597585501089506}\\                                
114688& \num{2.6272027065269014e-06}&\numeoc{2.934216981613506}& \num{3.1222838613630545e-07}&\numeoc{2.933446047312089}& \num{4.4961962451033915e-06}&\numeoc{2.973093720157347}& \num{4.706348523001752e-09}&\numeoc{3.9341186812908218}\\ 
\bottomrule
\end{tabular}
\caption{The $H^1$-seminorm error of the velocity, the $L^2$-norm error of the pressure and the stress and the $L^2$-norm error of the velocity for different polynomial orders $k=1,2,3$ for the three dimensional case and a fixed viscosity $\nu = 10^{-3}$} \label{threedexample}
\end{center}
\end{table}

We conclude with a few remarks on the cost of solving the discrete
system~\eqref{eq::discrmixedstressstokesweak}. After an element wise
static condensation step there are two different types of degrees of
freedom (dofs) that couple at element interfaces. These coupling dofs
determine the costs for the factorization step of the assembled
matrix.  In the $d=2$ case, the normal continuity of the
$H(\divergence)$-conforming velocity space demands $k+1$ dofs per
interface, while the normal-tangential continuity of the stress space
$\Stressspaceh$ requires $k$ dofs, i.e., we have $2k+1$ dofs per
interface.  This is comparable to the number of interface degrees of
freedom for standard methods.  In fact, it is identical to the number
of dofs per interface of an advanced method (with a reduced
stabilization called ``projected jumps'') presented in the recent work
of \cite{LS_CMAME_2016}. Similar cost comparison 
observations apply for the $d=3$ case.

\begin{figure}[h]
  \begin{center}
    \pgfplotstableread{probust/probust_output_2.out} \prob
\pgfplotstableread{probust/probust_output_TH_2.out} \probTH
\pgfplotstableread{probust/probust_output_3.out} \probthree
\pgfplotstableread{probust/probust_output_TH_3.out} \probTHthree
\pgfplotstableread{probust/probust_output_4.out} \probfour
\pgfplotstableread{probust/probust_output_TH_4.out} \probTHfour

\definecolor{myblue}{RGB}{62,146,255}
\definecolor{mygreen}{RGB}{22,135,118}
\definecolor{myred}{RGB}{255,145,0}

\begin{tikzpicture}
  [  
  scale=1
  ]
  \begin{axis}[
    name=plot1,
    scale=0.8,
    title = $|| \nabla \Velvar - \nabla \Velvarh ||_0$,
    legend entries={MCS, TH, $k=2$,$k=3$,$k=4$},
    legend style={text height=0.7em },
    legend style={draw=none},
    style={column sep=0.1cm},
    xlabel=$\nu$,
    x label style={at={(0.95,0.05)},anchor=west},
    xmode=log,
    ymode=log,
    ytick = {1e-6,1e-4,1e-2,1e-0,1e2,1e4},
    y tick label style={
      /pgf/number format/.cd,
      fixed,
      precision=2
    },
    xtick = {1e-10,1e-8,1e-6,1e-4,1e-2,1e-0,1e2},
    x tick label style={
      /pgf/number format/.cd,
      fixed,
      precision=2
    },
    grid=both,
    legend style={
      cells={align=left},
      at={(1.05,0.8)},
      anchor = north west
    },
    ]

    \addlegendimage{mark=*}
    \addlegendimage{mark=diamond*}

    \addlegendimage{magenta}
    \addlegendimage{myblue}
    \addlegendimage{myred}
    
    \addplot[line width=0.5pt, color=magenta, mark=*] table[x=0, y=1]{\prob};    
    \addplot[line width=0.5pt, color=magenta, mark=diamond*] table[x=0, y=1]{\probTH};

    \addplot[line width=0.5pt, color=myred, mark=*] table[x=0, y=1]{\probfour};
    \addplot[line width=0.5pt, color=myred, mark=diamond*] table[x=0, y=1]{\probTHfour};

    \addplot[line width=0.5pt, color=myblue, mark=*] table[x=0, y=1]{\probthree};
    \addplot[line width=0.5pt, color=myblue, mark=diamond*] table[x=0, y=1]{\probTHthree};
    
  \end{axis}
\end{tikzpicture}
  \end{center}
  \vspace{-0.7cm}
  \caption{The $H^1$-seminorm error for the MCS method and a Taylor-Hood approximation for $k=2,3,4$ and varying viscosity $\nu$.}
\label{fig:probustplot}
\end{figure}
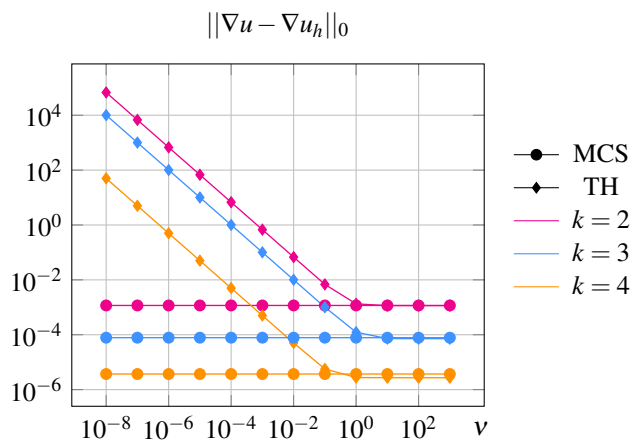

\section*{Acknowledgements}

 Philip L. Lederer has been funded by the Austrian Sicence Fund (FWF) through the research programm ``Taming complexity in partial differential systems'' (F65) - project ``Automated discretization in multiphysics'' (P10).
Part of this work was completed while 
two of the authors were at ``BCAM -- Basque Center for
Applied Mathematics.'' We gratefully acknowledge the hospitality of BCAM
and its funding through ``MINECO: BCAM Severo Ochoa excellence
accreditation  SEV-2013-0323.''

\bibliographystyle{mystyle-BIB}
\bibliography{literature}

\end{document}